\documentclass[11pt]{article}
\usepackage{amssymb}
\usepackage{mathrsfs}
\usepackage{latexsym,amsmath,amssymb,amsfonts,epsfig,graphicx,cite,psfrag}
\usepackage{eepic,color,colordvi,amscd}
\definecolor{blue}{rgb}{0,0,0.9}
\definecolor{red}{rgb}{0.9,0,0}
\definecolor{green}{rgb}{0,0.9,0}

\usepackage{ebezier}
\usepackage{verbatim}
\usepackage{subfigure}
\usepackage{enumitem}
\usepackage{algorithm2e}
\usepackage{dsfont}
\usepackage{amsthm}
\usepackage{comment}
\usepackage{tikz-cd}
\usepackage{color}
\usepackage{longtable}
\usepackage{hyperref}

\theoremstyle{thmstyleone}%
\newtheorem{theorem}{Theorem}
\newtheorem{proposition}[theorem]{Proposition}%
\newtheorem{assumption}[theorem]{Assumption}

\theoremstyle{thmstyletwo}%
\newtheorem{remark}{Remark}%
\usepackage{dsfont}
\theoremstyle{thmstylethree}%
\newtheorem{definition}{Definition}%
\newtheorem{lemma}{Lemma}%
\usepackage{longtable}

\def\H{\mathcal{H}}
\def\<{\left\langle}
\def\>{\right\rangle}

\def\E{\mathcal{E}}
\def\M{\mathcal{M}}
\def\T{\mathcal{T}}

\def\DD{{\rm Diag}}

\def\B{\mathcal{B}}
\def\C{\mathcal{C}}

\def\BH{\mathcal{BH}}
\def\I{\mathcal{I}}
\def\J{\mathcal{J}}
\def\R{\mathbb{R}}

\def\Re{{\rm Retr}}

\def\g{{\rm grad}}

\def\BG{\mathcal{BG}}

\def\V{\mathcal{V}}
\def\E{\mathcal{E}}
\def\S{\mathbb{S}}
\def\1{\mathds{1}}
\def\supp{{\rm supp}}
\def\Ran{{\rm Range}}
\def\Ker{{\rm Ker}}

\def\rr{{\rm rank}}
\def\dist{{\rm dist}}
\def\dd{{\rm diag}}
\def\P{{\rm Proj}}

\def\PA{\mathcal{P}_{A,b}}

\def\O{\mathcal{O}}

\def\({\left(}
\def\){\right)}

\def\Diag#1{{D}_{#1}}
\def\DD{{\rm Diag}}
\def\norm#1{\|#1\|}
\let\svthefootnote\thefootnote
\newcommand\blankfootnote[1]{%
	\let\thefootnote\relax\footnotetext{#1}%
	\let\thefootnote\svthefootnote%
}

\begin{document}
	\title{Optimization over convex polyhedra via Hadamard parametrizations}
	\author{Tianyun Tang
\thanks{Department of Mathematics, National 
         University of Singapore, Singapore 
         119076 ({\tt ttang@u.nus.edu}). 
         }, \quad 
	 Kim-Chuan Toh\thanks{Department of Mathematics, and Institute of 
Operations Research and Analytics, National 
         University of Singapore, 
       Singapore 
         119076 ({\tt mattohkc@nus.edu.sg}).  The research of this author is supported 
by the Ministry of Education, Singapore, under its Academic Research Fund Tier 3 grant call (MOE-2019-T3-1-010).} 
	 }
	\date{\today}
	\maketitle	


\begin{abstract}
In this paper, we study linearly constrained optimization problems (LCP).  After applying Hadamard parametrization, the feasible set of the parametrized problem (LCPH) becomes an algebraic variety, with conducive geometric properties which we explore in depth. We derive explicit formulas for the tangent cones and second-order tangent sets associated with the parametrized polyhedra. Based on these formulas, we develop a procedure to recover the Lagrangian multipliers associated with the constraints to verify the optimality conditions of the given primal variable without requiring additional constraint qualifications. Moreover, we develop a systematic way to stratify the variety into a disjoint union of finitely many Riemannian manifolds. This leads us to develop a hybrid algorithm combining Riemannian optimization and projected gradient to solve (LCP) with convergence guarantees. Numerical experiments are conducted to verify the effectiveness of our method compared with various state-of-the-art algorithms.
\end{abstract}

\bigskip
\noindent{\bf keywords:} convex polyhedra, Hadamard parametrization, algebraic variety, Riemannian optimization
\\[5pt]



\maketitle

\section{Introduction}\label{Sec-Intro}

\subsection{Optimization over convex polyhedra}\label{Subsec-LCP}
In this paper, we aim at exploring a new approach for solving the following optimization problem over a convex polyhedron:
\begin{equation}\label{LCP}
\min\left\{ \phi(x):\ Ax=b,\ x\in \R^n_+\right\}, \tag{LCP}
\end{equation}
where $\phi:\R^n\rightarrow \R$, $A=[A_1;A_2;\ldots;A_m]\in \R^{m\times n}$ and $b\in \R^m.$ 
We denote the feasible set of (\ref{LCP}) as  $\C_{A,b}$ and make the following basic assumption.

\begin{assumption}\label{ass1}
$\phi$ (not necessarily convex) is continuously differentiable;  $A$ has full row rank;
 $\C_{A,b}\neq \emptyset.$
\end{assumption}

Problem (\ref{LCP}) covers many types of optimization problems including linear programming (LP) and convex quadratic programming (CQP) \cite{LNP}. Naturally, it has wide range of applications such as portfolio optimization \cite{perold1984large} and machine learning \cite{rose1998deterministic}. 

Numerous algorithms have been developed to tackle (\ref{LCP}), as evidenced by various studies \cite{fletcher1972algorithm,powell1989tolerant,hager2023algorithm,di2023stationarity}. For specific types of objective functions, such as linear and convex quadratic functions, well-established techniques like simplex methods \cite{dantzig1951maximization} and interior point methods \cite{wright1997primal} have been comprehensively implemented in robust commercial solvers such as Gurobi \cite{gurobi} and Mosek \cite{mosek}. In addition to these classical approaches, contemporary algorithms that leverage the dual augmented Lagrangian method 
and semi-smooth Newton method \cite{liang2022qppal,li2020asymptotically,li2020efficient} have been shown to achieve asymptotic super-linear convergence. These methods have proven to be more efficient than commercial solvers for certain large scale structured problems. 
For specialized cases of convex polyhedra, including Birkhoff polytopes and transportation matrices, entropy-regularized methods \cite{yang2022bregman,cuturi2013sinkhorn} can offer the advantage of low per-iteration computational complexity to achieve a reasonable level of accuracy. While the aforementioned algorithms directly address (\ref{LCP}) by exploiting the properties of convexity and linearity inherent to convex polyhedra, this paper focuses on 
Hadamard parametrization for solving (\ref{LCP}) by transforming it into a nonlinear and nonconvex problem even when the original problem is convex.
 This approach will be further elaborated in the subsequent subsection.

\subsection{Hadamard parametrization and its limitations}\label{Subsec-Hadamard}
In (\ref{LCP}), the challenge of non-smoothness arises primarily from the nonnegativity constraint $x\in \R^n_+.$ One effective strategy for circumventing this constraint is through Hadamard parametrization that involves transforming $x$ into a squared variable $y\circ y,$ which leads us to the following modified problem with only equality constraints:
\begin{equation}\label{LCPH}
\min\left\{ f(y):=\phi(y\circ y):\ {A(y\circ y)}=b,\ y\in \R^n\right\}. \tag{LCPH}
\end{equation}
This transformation intrinsically preserves the non-negativity constraint on $x$. We define the feasible set of (\ref{LCPH}) as $\PA:=\left\{y\in \R^n:\ {A(y\circ y)}=b\right\},$ and referring to it as a parametrized polyhedron. Note that both (\ref{LCPH}) and (\ref{LCP}) share the same global optimal value. {Although $\PA$ doesn't involve inequality constraints, we must mention that it could still be non-smooth because it might contain singular points, whose tangent cone is not a linear space, i.e., $\PA$ might not be a smooth manifold.}

{The utilization of squared variables in optimization can be traced back to the 1960s  \cite{nash1998sumt}, with recent applications in Lasso-type problems \cite{hoff2017lasso,kolb2023edges}. More recently, this idea was further explored by Ding and Wright in \cite{ding2023squared} for general optimization problems, whose inequality constraints $c(x)\leq 0$ are converted into equality constraints of the form $c(x) + v\circ v = 0$ by
introducing squared slack variables. It is worth noting that when it comes to the  
problem (\ref{LCP}), their conversion simply add the constraint $x=v\circ v$ to the problem, which is slightly different from our case where the variable $x$ is directly replaced by $y\circ y$. 
At the moment, it is unclear whether the two approaches could show 
similar practical behaviour. } 

Despite its simplicity and intuitive appeal, {the usage of squared variables} 
has been subjected to considerable criticism for its inherent non-convexity and stringent regularity requirements \cite{armand2012squared,fukuda2017note}. {The former issue pertains to the case when $\phi$ is convex: whence 
the problem (\ref{LCP}) is a convex optimization problem but the new problem (\ref{LCPH}) is in general non-convex. Recent works on landscape analysis have observed that under the linearly independent constraint qualification (LICQ)  assumption \cite{numerical}, the smooth parametrizations of various optimization problems \cite{li2021simplex,Boumal2,BM2} have benign non-convexity. More generally, Levin et al., in their comprehensive study \cite{levin2022effect}, have identified conditions where a second-order stationary point in a smooth-parametrized problem also serves as a first-order stationary point of the original problem, which will be a global minimizer if the original problem is convex. Apart from the benign non-convexity, the LICQ assumption also implies that the feasible set of the parametrized problem is a smooth manifold, thus allowing us to use Riemannian optimization method \cite{absil2008optimization,intromani} to solve it with guaranteed convergence.
}

While the {non-convexity could be benign} under the LICQ assumption, the regularity issue is perceived as more serious {because of two issues. The first issue is about the optimality conditions of (\ref{LCPH}). The optimality condition of (\ref{LCP}), derived from Farka's lemma \cite{LNP}, does not require any regularity conditions. For a given $x\in \mathcal{C}_{A,b},$ checking the first order optimality condition is equivalent to an LP feasibility problem, which can be done in polynomial time. In contrast, the derivation of the optimality conditions of (\ref{LCPH}) is based on the implicit function theorem, which relies on the LICQ property. Note that the statement of the traditional KKT conditions for a given primal-dual solution $(x,\lambda)$ of (\ref{LCPH}) \cite[chapter 12]{numerical} 
though doesn't require the LICQ property but it 
requires the knowledge of the multiplier.
However, when it comes to landscape analysis \cite{levin2022effect}, we are only given the primal variable $y\in \PA$. In this case, as far as we know, how to recover the Lagrangian multiplier $\lambda$ to verify the optimality conditions of $y$ hasn't been studied yet. }

{The second issue is about algorithms to solve (\ref{LCPH}). Because $\mathcal{C}_{A,b}$ is a convex polyhedra, traditional algorithms like barrier method \cite[chapter 17]{numerical} and projected gradient method \cite{birgin2000nonmonotone} can be applied to solve it without LICQ condition. However, when it comes to (\ref{LCPH}), the commonly used algorithms like sequential quadratic programming \cite[chapter 18]{numerical} and Riemannian optimization method \cite{li2021simplex,xiao2023dissolving} rely on the LICQ property, which may not always hold for real-world problems. } 

In this paper, we aim to address the above two issues by conducting a comprehensive analysis of the (\ref{LCPH}) problem. The details of our contributions will be elaborated 
in the following subsections.

\subsection{Optimality conditions}\label{Subsec-opt}
{
If the LICQ property holds at some $y\in \PA,$ then $\PA$ is a smooth manifold in a neighbourhood of $y.$ In this case, we can easily recover the Lagrangian multiplier $\lambda\in\R^m$ by computing the projection of $\nabla f(y)$ onto the tangent space ${\rm T}_y \PA,$ which is equivalent to solving a positive definite linear system. After obtaining $\lambda,$ we can move on to verify the first order and second order KKT conditions, which is equivalent to that the Riemannian gradient is zero and the Riemannian Hessian is positive semidefinite \cite[chapter 4-6]{intromani}. The violation of each of the conditions will produce a direction to decrease the objective function value. 

However, when the LICQ property doesn't hold at $y\in \PA,$ its tangent cone may not be a linear space and the above approach doesn't work. In order to verify the optimality of a given singular point $y\in \PA$, in Section~\ref{Sec-optcon}, we will derive explicit formulas for both the tangent cone and the second-order tangent set of $y.$ With this information, we develop a systematic way to recover the Lagrangian multiplier and verify the first and second-order necessary conditions for (\ref{LCPH}). One corollary of our result is that the ``$2 \Rightarrow 1$" criterion in \cite{levin2022effect} holds for (\ref{LCP}) and (\ref{LCPH}). Here, instead of using the traditional KKT conditions \cite[chapter 12]{numerical}, we choose the more geometric optimality conditions from \cite[chapter 3]{Perturbation}, which are properties of the primal variable $y\in \PA$ itself and the conditions are always necessary. Thus, it is more suitable for landscape analysis and can be used to design local search algorithms. However, when it comes to the traditional KKT conditions, they are in general not necessary without the LICQ property. In this case, we can say nothing about the optimality of the primal-dual pair $(y,\lambda),$ even if they satisfy the traditional KKT conditions.}

\subsection{Riemannian optimization on algebraic variety}\label{Subsec-Rie} 
Riemannian optimization is a widely-used technique for solving smooth parametrized problems. However, directly applying the Riemannian optimization approach to (\ref{LCPH}) poses challenges, as the feasible set $\PA$ is an algebraic variety that may contain singular points. Although manifold optimization is well-understood and supported by a variety of solvers \cite{boumal2014manopt,xiao2022cdopt}, optimization over algebraic varieties remains an emerging area of study \cite{levin2020towards,levin2022finding,olikier2023apocalypse}. 

{When extending Riemannian optimization method to algebraic variety, a pivotal insight is that every real algebraic variety can be decomposed into finite smooth submanifolds \cite{whitney1992elementary}. With such decomposition, we can apply Riemannian optimization method on one of the manifolds and switch to another one when the iterations approach the boundary. This idea has been used in the determinantal variety $\R^{m\times n}_{\leq r}:=\left\{ X\in \R^{m\times n}:\ \rr(X)\leq r \right\}$ \cite{rankadap,lee2022escaping} and some special low-rank SDP problems, whose singular points are exactly rank-1 solutions \cite{tang2023feasible,tang2023solving}.} 

{In Section~\ref{Sec-alg}, we will adopt the above idea to solve (\ref{LCPH}). In order to apply Riemannian optimization, we will first stratify $\PA$ according to the maximal linearly independent set of active constraints.} Given the singularity of $\PA$, these submanifolds may not be closed, thus potentially causing Riemannian optimization methods to fail at convergence when the iterates approach a submanifold 
boundary -- a phenomenon termed $apocalypses$ \cite{levin2022finding}. To 
overcome this difficulty, we design a globally convergent hybrid method, which combines Riemannian optimization with the projected gradient method. The projected gradient method allows the iteration point to jump from one submanifold to another 
while decreasing the function value. On the other hand,
the Riemannian gradient descent can efficiently decrease the function value on the selected submanifold by exploiting its underlying sparsity structure and smoothness. {It is worth mentioning that recent works \cite{pauwels2024generic,olikier2024projected} show that applying projected gradient directly to (LCPH) don't suffer from the $apocalypses$ issue. However, it could be complicated to computed the metric projection onto the non-convex set $\PA.$ {On the contrary}, when using Riemannian optimization method, we can apply some iterative methods to compute the retraction for general nonlinear constraints as long as LICQ holds. We will present this in detail in Section~\ref{Sec-alg}.}

\subsection{Organization of this paper}
In the following subsection, we  provide some {notation} and definitions that will be used throughout this paper. In Section~\ref{Sec-optcon}, we  study the geometric properties of $\PA$ as well as {how to verify the} optimality conditions of (\ref{LCP}) and (\ref{LCPH}). In Section~\ref{sec-manidcmp}, we  study the manifold decomposition of $\PA.$ In Section~\ref{Sec-alg}, we present
 our algorithm and provide its convergence analysis. In Section~\ref{Sec-numer}, we  conduct numerical experiments to verify the efficiency of our algorithm. We end this paper with a conclusion in Section~\ref{Sec-conc}. 
 For smoother reading of the paper, we put some lemmas and proofs in the appendix.

\subsection{Basic notation and definitions}

\begin{itemize}
\item We use $\|\cdot\|$ to denote the Euclidean norm or the Frobenius norm.

\item For any $m,n\in\mathbb{N}^+,$ $\R^{m\times n}_+$ is the set of $m$ by $n$ nonnegative matrices. $\S^n$ is the set of $n$ by $n$ symmetric matrices. $\S^n_+$ is the set of $n$ by $n$ symmetric positive semidefinite matrices.

\item Given two matrices $M,N$ having the same number of columns, $[M\,;\,N]$ denotes
the matrix obtained by appending $N$ to the last row of $M$.

\item For any $A\in \R^{m\times n}$ and $\I\subset [m],\J\subset [n]$,
$A_{\I,:}\in \R^{|\I|\times n}$ represents the submatrix of $A$ that consists of the rows indexed by $\I$;
 $A_{:,\J}\in \R^{m\times |\J|}$ represents the submatrix of $A$ that consists of the columns indexed by $\J$; 
 $A_{\I,\J}$ represents the submatrix of $A$ that consists of rows indexed by $\I$ and columns indexed by $\J.$

\item For any $y\in \R^n,$  $\supp(y):=\left\{ i\in [n]:\ y_i\neq 0\right\}$ denotes the support of $y.$  Let $\1_y\in \R^n$ be the vector whose entries in $\supp(y)$ are 1's and entries outside $\supp(y)$ are 0's. We use ${\bf 1}_n$ and ${\bf 0}_n$ to denote the all ones vector and zero vector, respectively. For any $i\in [n],$ let $e_i$ be the $i$th column vector of the $n$ by $n$ identity matrix. We omit the length of $e_i$ when it is clear.

\item For any $y\in \R^n,$ $\DD(y)\in \R^{n\times n}$ denotes the diagonal matrix with its diagonal entries given by $y.$ We will sometimes also use $\Diag{y}$ to denote $\DD(y)$
to simplify the notation. 
For any $A\in \R^{n\times n},$ $\dd(A)\in \R^n$ is the vector of the diagonal entries of $A.$

\item For any matrix $A\in \R^{m\times n},$ we use ${\Ran[A]}:=\left\{Ax:\ x\in \R^n\right\}$ to denote the range space of $A,$ and $\Ker [A]:=\left\{x\in \R^n:\ Ax=0\right\}$ to denote the null space of $A.$

\item For any $y\in \PA,$ we call it a regular (smooth) point if the LICQ condition holds at $y$ i.e., $\{A_1^\top \circ y,\ldots,A_m^\top \circ y\}$ is a linearly independent set of vectors. Otherwise, we call $y$ a non-regular (singular) point. 

\item For any $r>0$ and $y\in \R^n,$ $B_r(y):=\left\{ z\in\R^n:\ \|z-y\|<r \right\}.$

\item
Let $S$ be a closed subset of $\R^n$ and $x\in S,$ the inner and contingent tangent cones \cite[chapter 2]{Perturbation} of $x$ are defined as follows:
\begin{equation}\label{ITcone}
\T^i_S(x):=\left\{h\in \R^n:\ \dist\(x+th,S\)=o(t),\ t\geq 0\right\}. 
\end{equation}
\begin{equation}\label{CTcone}
\T_S(x):=\left\{h\in \R^n:\ \exists\ t_k\downarrow 0,\ \dist\(x+t_kh,S\)=o(t_k)\right\}. 
\end{equation}

\item 
Let $S$ be a closed subset of $\R^n,$ $x\in S$ and $h\in \R^n,$ the inner and outer second-order tangent sets \cite[chapter 3]{Perturbation} of $(x,h)$ are defined as follows:
\begin{equation}\label{I2tset}
\T^{i,2}_S(x,h):=\left\{w\in \R^n:\ \dist\(x+th+{t^2w}/{2},S\)=o(t^2),\ t\geq 0\right\}.
\end{equation}
\begin{equation}\label{O2tset}
\T^2_S(x,h):=\left\{w\in \R^n:\ \exists\ t_k\downarrow 0,\ \dist\(x+t_kh+{t_k^2w}/{2},S\)=o(t_k^2)\right\}.
\end{equation}

\item For any closed convex set $\C$, $\P_\C(x):=\arg\min\left\{\|y-x\|:\ y\in \C\right\}$ is the orthogonal projection mapping.

\end{itemize}


\section{Optimality conditions}\label{Sec-optcon} 
In this section, we will derive the optimality conditions of (\ref{LCPH}) as well as its landscape analysis. Before we derive the optimality conditions, we first study the geometric properties of the algebraic variety $\PA$ in the following subsection. 

\subsection{Geometric properties of $\PA$}\label{subsec-geo}

\begin{definition}\label{Regrep}
For any $y\in \PA,$ let $[M;N]\in \R^{m\times m}$ be a full rank matrix such that the rows of $M$ is a basis of $\Ran[A\Diag{y}]$ and the rows of $N$ is a basis of its orthogonal complement in $\R^m$ (hence
$MN^\top=0$).
It is easy to see that $MA\Diag{y}$ has full row rank, $NA\Diag{y}=0$ and $Nb=0.$ Let $\rr(M)=\rr(A\Diag{y})=r.$ We call $(M,N,r)$ a regular representation of $y$ in $\PA.$ Note that by definition, for any $h\in \R^n$, $A\Diag{y} h = M^\top z$ for an unique $z\in \R^r$.
\end{definition}

Note also that the regular representation for a given $y\in \PA$ is not unique because the basis vectors of a linear space can be non-unique. Also, $y\in \PA$ is a regular point if and only if $r=m,$ in which case $N$ is a null matrix. The regular representation of $y\in \PA$ is only related to $\supp(y)$ because ${\Ran[A\Diag{y}]=\Ran[A\Diag{z}]}$ if $\supp(y)=\supp(z).$ The following two propositions provide the analytic formulas of the tangent cone and second-order tangent set of $\PA.$ We put their proofs in appendix.

\begin{proposition}\label{Prop-tcone}
For any $y\in \PA$ with regular representation $(M,N,r),$ the inner and contingent tangent cones of $y$ are given as follows:
\begin{equation}\label{PAtcone}
\T_{\PA}^i(y)=\T_{\PA}(y)=\widehat{\T}(y):=\left\{h\in \R^n:\ MA \(y\circ h\)=0,\ NA(h\circ h)=0\right\}.
\end{equation}
\end{proposition}

\begin{proposition}\label{Prop-2tset} 
Consider $y\in \PA$ with regular representation $(M,N,r)$ and $h\in \T_{\PA}(x).$ 
Let $[V;W]\in \R^{(m-r)\times (m-r)}$ be a full rank matrix such that the rows of $V$ is a basis for ${\Ran[NA\Diag{h}]}$ and $VW^\top=0.$ Then the inner and outer second-order tangent sets of $(y,h)$ are given as follows:
\begin{multline}
\T^{i,2}_{\PA}(y,h)=\T^2_{\PA}(y,h)=\widehat{\T^2}(y,h):=\big\{ w\in \R^n:\ MA \(y\circ w+h\circ h\)=0,\\ VNA(h\circ w)=0,\ 
 WNA\(w\circ w\)=0\big\} \label{PA2tset}.
\end{multline}
\end{proposition}

\begin{remark}\label{2nddd}
Proposition~\ref{Prop-2tset} implies that $\PA$ is second-order directionally differentiable at any point $y\in \PA$ along any direction $h\in \T_{\PA}(y)$ \cite[definition 3.32]{Perturbation}. Note that the set $\widehat{\T}(y)$ in Proposition \ref{Prop-tcone}
 is independent of the representation $(M,N,r)$. That is,
if $(M',N',r)$ is another representation of $y\in \PA$, then $\widehat{\T}(y) = \widehat{\T'}(y) :=
\{ h\in \R^n : M'A(y\circ h) =0, N' A(h\circ h) = 0\}.$ A similar remark also holds
for the set $\widehat{\T^2}(y,h)$ in Proposition  \ref{Prop-2tset}.

\end{remark}


\subsection{Optimality conditions of (\ref{LCP}) and (\ref{LCPH})}\label{subsec-opt}
With the formula of tangent cones and second-order tangent sets of $\PA,$ we are now able to derive the optimality conditions. Since the feasible set of (\ref{LCP}) is a convex polyhedral set, the linear constraint qualification (LCQ) holds and we don't need other regularity assumptions. The first-order necessary conditions of (\ref{LCP}) are as follows.

\begin{definition}\label{LCP1stKKT} [first-order necessary conditions]
$x\in \R^n$ is called a first-order stationary point of (\ref{LCP}) if there exists $\lambda\in \R^m$ such that
\begin{equation}\label{LCPKKT}
(1)\ A x=b,\ x\in \R^n_+,\ (2)\ \nabla \phi(x)-A^\top \lambda\in \R^n_+,\ (3)\ \(\nabla \phi(x)-A^\top \lambda\)\circ x=0.
\end{equation}
We refer to the above three conditions as primal feasibility, dual feasibility and complementarity, respectively. 
\end{definition}
We do not consider second-order necessary condition of (\ref{LCP}) because verifying this property is in general NP-hard \cite{quadNPhard}. Also, for many  applications of (\ref{LCP}), $\phi$ is a convex function so that a first-order {stationary} point is {already} a global {minimizer}. 

Now, we move on to study the optimality conditions of (\ref{LCPH}). {As mentioned in Subsection~\ref{Subsec-opt}, instead of using the traditional KKT conditions \cite[chapter 12]{numerical}, we use the following more geometric optimality conditions \cite[chapter 3]{Perturbation}.}

\begin{definition}\label{LCPN1stKKT} [first-order necessary conditions]
$y\in \R^n$ is called a first-order stationary point of (\ref{LCPH}) if $y\in \PA$ and $\nabla f(y)^\top h\geq 0,$ for any $h\in \T_{\PA}(y).$ 
\end{definition}

\begin{definition}\label{LCPN2ndKKT} [second-order necessary conditions]
$y\in \R^n$ is called a second-order stationary point of (\ref{LCPH}) if $y$ is a first-order stationary point, and $\nabla f(y)^\top w+h^\top \nabla^2 f(y) h\geq 0$ for any $h\in \T_{\PA}(y)$ satisfying $\nabla f(y)^\top h=0$ and any $w\in \T^2_{\PA}(y,h).$
\end{definition}

{Although the above two definitions are purely geometric, from the formulas in Proposition~\ref{Prop-tcone} and \ref{Prop-2tset}, we can obtain the algebraic characterization of stationary points. We state the results in} the following two lemmas, which offer reformulations of the above two optimality conditions to make them easily verifiable.

\begin{lemma}\label{1stlem}
Suppose $y\in \PA$ has regular representation $(M,N,r),$ $x=y\circ y.$ Then $y$ is a first-order stationary point of (\ref{LCPH}) if and only if there exists $\lambda\in \R^{r}$ such that 
\begin{equation}\label{Oct_28_1}
\(\nabla \phi(x)-A^\top M^\top \lambda\)\circ y=0,
\end{equation}
which is equivalent to $\nabla \phi(x)\circ y\in \Ran[ \Diag{y}A^\top ],$ i.e., $\P_{\Ker[A\Diag{y}]}\( \nabla\phi(x)\circ y \)=0.$
\end{lemma}

\begin{proof}
\noindent{\bf Necessity.} Suppose $y\in \PA$ is a first-order stationary point. For any $h\in \R^n$ such that $MA(y\circ h)=0,$ we have that $MA(y\circ (h\circ \1_y))=0.$ Because $NA\Diag{y}=0,$ we have that $NA\((h\circ \1_y)\circ (h\circ \1_y)\)=0.$ Thus, $h\circ\1_y\in \T_{\PA}(y).$ From Definition~\ref{LCPN1stKKT}, we have that $\<\nabla \phi(x)\circ y,h\circ \1_y\>\geq 0,$ which implies $\<\nabla \phi(x)\circ y,h\>\geq 0.$ Since this inequality holds for any $h\in \R^n$ such that $MA(y\circ h)=0,$ we have that $\nabla \phi(x)\circ y\in \Ran[ \Diag{y}A^\top M^\top ]=\Ran[ \Diag{y}A^\top ].$ Thus, there exists $\lambda\in \R^{r}$ such that (\ref{Oct_28_1}) holds.

\medskip

\noindent{\bf Sufficiency.} Suppose (\ref{Oct_28_1}) holds. Then for any $h\in \T_{\PA}(y),$ we have that 

$$\<\nabla \phi(x)\circ y,h\>=\<\(A^\top M^\top \lambda\)\circ y,h\>=\<\lambda,MA(y\circ h)\>=0,$$ where the last equality follows from the fact that $MA(y\circ h)=0$ 
for $h\in\T_{\PA}(y)$ (see Proposition~\ref{Prop-tcone}).
\end{proof}

\begin{lemma}\label{2ndlem}
Suppose $y\in \PA$ has regular representation $(M,N,r)$ and $x=y\circ y.$ Then $y$ is a second-order stationary point of (\ref{LCPH}) if and only if there exists $\lambda\in \R^{m}$ such that (\ref{LCPKKT}) holds and 
\begin{equation}\label{Oct_28_6}
(d\circ \1_y)^\top \nabla^2 \phi(x) (d\circ \1_y)\geq 0,
\end{equation}
for any $d\in \R^n$ satisfying $A(d\circ \1_y)=0.$
\end{lemma}

\begin{proof}
\bigskip

\noindent{\bf Necessity.} From Proposition~\ref{Prop-2tset}, Definition~\ref{LCPN2ndKKT} and Lemma~\ref{1stlem}, we know that $y\in \PA$ is a second-order stationary point if and only if there exists $\lambda'\in \R^{r}$ such that (\ref{Oct_28_1}) and the following condition holds:
\begin{equation}\label{Oct_28_7}
\<\nabla \phi(x),h\circ h+y\circ w\>+2(y\circ h)^\top \nabla^2 \phi(x)(y\circ h)\geq 0,
\end{equation} 
for any $h\in \T_{\PA}(y)$ such that $\nabla f(y)^\top h=0$ and $w\in \T^2_{\PA}(y,h).$ 

For any $h\in \T_{\PA}(y),$ we have that $MA(y\circ h)=0$ and $NA(h\circ h)=0.$ Since
the mapping $w\rightarrow MA(y\circ w)$ is surjective, we may choose $w\in \R^n$ such that $MA(y\circ w+h\circ h)=0$ and $\supp(w)\subset \supp(y).$ Because $w=w\circ \1_y$ and $NA\DD(\1_y) = 0$, 
 we know that $VNA(h\circ w)=0$ and $WNA(w\circ w)=0.$ Thus, $w\in \T^2_{\PA}(y,h).$ From (\ref{Oct_28_1}) and $MA\(y\circ h\)=0$, we have that
\begin{equation}\label{Oct_28_8}
\nabla f(y)^\top h=\< 2\nabla \phi(x)\circ y,h \>=2\<\(A^\top M^\top\lambda'\)\circ y,h\>=2\<\lambda',MA\(y\circ h\)\>=0.
\end{equation}
Therefore, $\(h,w\)$ satisfies all the conditions of (\ref{Oct_28_7}). We get the following inequality
\begin{eqnarray}
0 &\leq& \<\nabla \phi(x),h\circ h+y\circ w\>+2(y\circ h)^\top \nabla^2 \phi(x)(y\circ h)
 \nonumber \\
&=&\<\nabla \phi(x),h\circ h\>+\<\(A^\top M^\top \lambda'\)\circ y,w \>+2(y\circ h)^\top \nabla^2 \phi(x)(y\circ h) 
 \nonumber \\ 
&=&\<\nabla \phi(x),h\circ h\>+\<\lambda',MA\(y\circ w\) \>+2(y\circ h)^\top \nabla^2 \phi(x)(y\circ h)
 \nonumber \\  
&=&\<\nabla \phi(x),h\circ h\>-\<\lambda',MA\(h\circ h\) \>+2(y\circ h)^\top \nabla^2 \phi(x)(y\circ h)
 \nonumber \\ 
&=& \<\nabla \phi(x)-A^\top M^\top \lambda',h\circ h\>+2(y\circ h)^\top \nabla^2 \phi(x)(y\circ h),
\label{Oct_28_9}
\end{eqnarray} 
where the first equality comes from (\ref{Oct_28_1}) and the third equality comes from that $MA(y\circ w+h\circ h)=0.$ Now we are going to choose a specific $h\in \T_{\PA}(y)$ to deduce (\ref{Oct_28_6}).

First, consider any $d\in \R^n$ such that $A(d\circ \1_y)=0.$ We have that $d\circ \1_y=y\circ h'$ for some vector $h'\in \R^n$ such that $\supp(h')\subset\supp(y).$ Because $MA(y\circ h')=0$ and $NA(h'\circ h')=NA(h'\circ \1_y \circ h'\circ \1_y)=0.$ We have that $h'\in \T_{\PA}(y).$ Substituting $h'$ into (\ref{Oct_28_9}) and using (\ref{Oct_28_1}), we have that $\<\nabla \phi(x)-A^\top M^\top \lambda',h'\circ h'\>=0$ and hence $2(y\circ h')^\top \nabla^2 \phi(x)(y\circ h')\geq 0.$ Since $d\circ \1_y=y\circ h',$ we have that $2(d\circ \1_y)^\top \nabla^2 \phi(x)(d\circ \1_y)\geq 0,$ which verifies (\ref{Oct_28_6}).

Next, we will construct $\lambda\in \R^m$ such that \eqref{LCPKKT} holds. Consider any $z\in \R^n_+$ such that $NAz=0.$ We decompose $z=z'+z''$ such that $\supp(z')\subset \supp(y)$ and $\supp(z'')\cap \supp(y)=\emptyset.$ It is easy to see that $z',z''\in \R^n_+.$ Because $NA\Diag{y}=0,$ we have that $NA z'=0$ and so $NA (\sqrt{z''}\circ \sqrt{z''}) = NA z''=NA z-NA z'=0.$ 
Because $\supp(z'')\cap \supp(y)=\emptyset,$ we have that $y\circ z''=0$ and $MA(y\circ \sqrt{z''})=0.$ Thus, $\sqrt{z''}\in \T_{\PA}(y).$ 
Substituting $\sqrt{z''}$ into (\ref{Oct_28_9}) as ``$h$" and using the fact that $y\circ z''=0,$ we get $\<\nabla \phi(x)-A^\top M^\top \lambda',z''\>\geq 0.$ From $\supp(z')\subset \supp(y)$ and (\ref{Oct_28_1}), we have that $\<\nabla \phi(x)-A^\top M^\top \lambda',z'\>= 0.$ This two results together imply that $\<\nabla \phi(x)-A^\top M^\top \lambda',z\>\geq  0.$ Since this holds for any $z\in \R^n_+$ such that $NAz=0,$ from Farka's lemma \cite[section 6.3]{LNP}, there exists $\lambda''\in \R^{m-r}$ such that 
\begin{equation}\label{23_Sept18_1}
\nabla \phi(x)-A^\top M^\top \lambda'-A^\top N^\top \lambda''\in \R^n_+.
\end{equation}
From (\ref{Oct_28_1}) and $N A \Diag{y}=0,$ we have that $\(\nabla \phi(x)-A^\top M^\top \lambda'-A^\top N^\top \lambda''\)\circ y=0.$ By choosing $\lambda=M^\top \lambda'+N^\top \lambda'',$ we know that (\ref{LCPKKT}) is satisfied.

\medskip

\noindent{\bf Sufficiency.} From (\ref{LCPKKT}), we know that $\(\nabla \phi(x)-A^\top \lambda\)\circ y=0.$ Because $[M;N]$ is full rank, without loss of generality, we write $\lambda=M^\top \lambda'+N^\top \lambda''.$ 
Because $NA\Diag{y}=0,$ we have that $\(\nabla \phi(x)- A^\top M^\top \lambda'\)\circ y=0.$ Therefore, from Lemma~\ref{1stlem}, $y$ is a first-order stationary point of (\ref{LCPH}). For any $h\in \T_{\PA}(y)$ satisfying $\nabla f(y)^\top h=0$ and any $w\in \T^2_{\PA}(y,h),$ we have that
\begin{eqnarray}
&& \hspace{-7mm}\<\nabla \phi(x),h\circ h+y\circ w\>=\<\nabla \phi(x),h\circ h\>+\< \( A^\top M^\top\lambda'\)\circ y,w \> 
\nonumber\\
&& \hspace{-7mm}=\<\nabla \phi(x),h\circ h\>+\< \lambda',MA\(y\circ w\) \>
\nonumber \\
&&\hspace{-7mm}=\<\nabla \phi(x),h\circ h\>-\< \lambda',MA\(h\circ h\) \>-\< \lambda'',NA\(h\circ h\) \>\\
&&\hspace{-7mm}=\<\nabla \phi(x)-A^\top \lambda,h\circ h\>\geq 0,\qquad
\label{Oct_29_1}
\end{eqnarray}
where the first equality comes from (\ref{Oct_28_1}), the third equality comes from (\ref{PAtcone}) and (\ref{PA2tset}). 
Note that the last inequality holds because $\nabla \phi(x) - A^\top\lambda \geq 0.$
From (\ref{Oct_28_6}) and $A(y\circ h)=0$, we have that $(y\circ h)^\top \nabla^2 \phi(x)(y\circ h)\geq 0$. This together with (\ref{Oct_29_1}) implies that
\begin{equation}\label{Oct_29_3}
\nabla f(y)^\top w+h^\top \nabla^2 f(y)h=2\<\nabla \phi(x),h\circ h+y\circ w\>+4(y\circ h)^\top \nabla^2 \phi(x)(y\circ h)\geq 0.
\end{equation}
Therefore, $y$ is a second-order stationary point of (\ref{LCPH}).
\end{proof}

{The above two lemmas together imply the following result about the non-convex landscape of (\ref{LCPH}).}

\begin{proposition}\label{2ndprop}
If $y\in \PA$ is a second-order stationary point of (\ref{LCPH}), then $x=y\circ y$ is a first-order stationary point of (\ref{LCP}).
\end{proposition}
\begin{proof}
From Lemma~\ref{2ndlem}, we know that $x$ satisfies (\ref{LCPKKT}) in Definition~\ref{LCP1stKKT}. Thus we know that $x$ is a first-order stationary point of (\ref{LCP}). 
\end{proof}

\begin{remark}\label{w2ndrem}
Note that from Lemma~\ref{2ndlem}, a second-order stationary point of (\ref{LCPH}) is exactly a {\bf weak} second-order critical point of (\ref{LCP}) that was introduced by N. {Gould} \cite[subsection 4.2]{gould2006introduction}. In the next subsection, we illustrate how to verify the first-order and second-order optimality conditions.
\end{remark}
\subsection{Verifying optimality conditions}\label{subsec-verify}
Given a vector $y\in \R^n,$ its primal feasibility can be quickly confirmed by evaluating 
$\|A (y\circ y)-b\|$ to check if it is (numerically) zero. Verifying condition (\ref{Oct_28_6}) boils down to assessing the positive definiteness of a matrix with dimension $|\supp(y)|\times |\supp(y)|$, a task that can be performed in polynomial time. The more challenging aspects lie in the first-order condition (\ref{Oct_28_1}), dual feasibility, and complementarity of (\ref{LCPKKT}), primarily because the multiplier $\lambda$ is unknown. 
Remarkably, if  $y\in \PA$ is smooth\footnote{In this case, one can choose $M=I$ and {condition} (\ref{Oct_28_1}) is equivalent to the complementarity.}, then the linear system in the complementarity from (\ref{LCPKKT}) yields a unique least-squares solution $\lambda$ and one only have to check the dual feasibility using the multiplier $\lambda$. However, when $y$ is singular, the system {has} multiple least-squares solutions, which makes the determination of the correct $\lambda$ far from being straightforward. Actually, a method for recovering the multiplier is hidden within the proofs of Lemma~\ref{1stlem} and \ref{2ndlem}, which guide us to partition the multiplier into two components and to recover them sequentially.

Given $y\in \PA$ and $x=y\circ y$, from Lemma~\ref{LDL},  we have the LDL decomposition $A\Diag{x}A^\top=L\DD(d)L^\top,$ where $d\in \R^m_+$ and $L\in \R^{m\times m}$ is a lower triangular matrix such that $\dd(L)=e$. This computation can be executed using the {\sc Matlab} function $\texttt{ldl}.$ { It is worth noting that the {\sc Matlab} function $\texttt{ldl}$ typically performs a permutation on the matrix to reduce fill-ins. But 
this won't affect our manifold decomposition because the permutation only changes the order of affine constraints in $Ax=b,$ which is not related to the geometry of $\PA.$} It's worth noting that $y$ is regular if and only if $A\Diag{x}A^\top$ is positive definite, which is equivalent to $d>0$. 

\bigskip

\noindent{\bf Verifying the first-order stationarity.}

\medskip

\noindent From Lemma~\ref{1stlem}, we know that checking the first-order stationarity of $y$ is equivalent to checking that the orthogonal  projection of $\phi(x)\circ y$ onto 
 $\Ker[A\Diag{y}]$ is zero. The projection has the following explicit formula
\begin{equation}\label{eqproj}
\P_{\Ker[A\Diag{y}]} \(\nabla\phi(x)\circ y\)=\nabla\phi(x)\circ y-\Diag{y}A^\top \lambda,
\end{equation}
such that $\lambda\in \R^m$ is a solution of the following positive semidefinite linear system
\begin{equation}\label{linsys}
L\DD(d)L^\top \lambda=A\Diag{x}A^\top \lambda=A\Diag{y} (\nabla \phi(x)\circ y),
\end{equation}
which has a solution because the right hand side lies in the range space of $A\Diag{y}.$ The above linear system has a unique solution if and only if $y$ is regular. One particular solution of (\ref{linsys}) is given by
\begin{equation}\label{linsoln}
\lambda'=L^{-\top }\DD(d)^+L^{-1}A\Diag{y} (\nabla \phi(x)\circ y),
\end{equation}
where $\DD(d)^+$ is the pseudoinverse of $\DD(d)$ by taking the reciprocal of the nonzero diagonal elements. The solution set of (\ref{linsys}) is given as follows:
\begin{equation*}\label{linsolnset}
{\rm Soln}(y):=\left\{ \lambda'+\lambda'':\ L\DD(d)L^\top\lambda''=0,\ \lambda''\in \R^m \right\}.
\end{equation*}

\medskip

\noindent{\bf  Verify complementarity and dual feasibility.}

\medskip

\noindent 
From (\ref{eqproj}) and (\ref{linsys}), we know that if the first-order necessary condition holds, then the solution set ${\rm Soln}(y)$ is the set of multipliers satisfying the complementarity (\ref{LCPKKT}). The remaining task is to check whether there exists a multiplier inside ${\rm Soln}(y)$ satisfying the dual feasibility (\ref{LCPKKT}), which is equivalent to checking whether the optimal value of the following optimal problem is zero:
\begin{equation}\label{drecov0}
\min\left\{ \frac{1}{2} \left\|\P_{\R^n_-}\(\nabla \phi(x)-A^\top\lambda'-A^\top \lambda''\)\right\|^2
:\ L\DD(d)L^\top\lambda''=0,\ \lambda''\in \R^m \right\}.
\end{equation}
From the definition of regular representation $(M,N,r)$ in (\ref{Regrep}), one particular choice of $(M,N)$ is $M=I_{\supp(d),:}L^\top$ and $N=I_{[m]\setminus \supp(d),:}L^{-1}.$  Problem (\ref{drecov0}) can be further simplified as
\begin{equation}\label{drecov}
\min\left\{\frac{1}{2}\left\| \P_{\R^n_-} \(\nabla \phi(x)-A^\top\lambda'-A^\top N^\top \mu\)\right\|^2 \;:\ \mu\in \R^{m-r} \right\},
\end{equation}
whose dual problem is as follows:
\begin{equation}\label{drecovd}
\max\left\{ -\< \nabla \phi(x)-A^\top\lambda',{z} \>-\frac{1}{2}\|{z}\|^2:\  NA{z}=0,\ {z}\in \R^n_+ \right\}.
\end{equation}
Problem (\ref{drecov}) is an unconstrained optimization problem whose objective function is convex and Lipschitz-continuously differentiable. We can use the (regularized) semi-smooth Newton method in \cite{li2020efficient} to solve it efficiently. Notably, in practice, $m-r$ is usually much smaller than $m$, so problem (\ref{drecov}) is usually easy to solve. Problem (\ref{drecovd}) is maximizing a strongly-concave function in a convex polyhedron, which can be done by an interior point method in polynomial time. {The following proposition says that when the dual feasibility fails, the solution of (\ref{drecovd}) will give us a descent direction.

\medskip
\begin{proposition}\label{escprop}
Suppose $\lambda'$ is defined as in (\ref{linsoln}) and $z=h\circ h$ is a feasible solution of (\ref{drecovd}) with a positive function value. Let $w$ be the least square solution of the linear system $MA\(y\circ w+h\circ h\)=0$ and $d:=y\circ w+h\circ h.$ We have that $\<\nabla \phi(x),d \><0$ and for $t>0$ sufficient small, $x+t\cdot d\in \mathcal{C}_{A,b}.$
\end{proposition}

\begin{proof}
From (\ref{linsoln}), $\lambda'$ is the least square solution of the linear system (\ref{linsys}). This implies that $\lambda'=M^\top \gamma$ for some $\gamma\in \R^r.$ We have that
\begin{multline}
\< \nabla \phi(x),h\circ h+y\circ w \>=\< \nabla \phi(x)-A^\top M^\top \gamma,h\circ h+y\circ w \>\\
=\< \nabla \phi(x)-A^\top \lambda',z \>+\< \(\nabla \phi(x)-A^\top \lambda'\)\circ y, w \>=\< \nabla \phi(x)-A^\top \lambda',z \><0,
\end{multline}
where the first equality comes from $MA\(y\circ w+h\circ h\)=0,$ the third equality comes from (\ref{eqproj}) and the fact that $w$ is the least square solution of the linear system $MA\(y\circ w+h\circ h\)=0,$ which implies that $w\in {\rm Ker}[AD_y]^\perp$ and the inequality follows from that the objective function value of (\ref{drecovd}) at $z$ is positive. Because $MA\(y\circ w+h\circ h\)=0,$ $NA D_y=0$ and $NA\(h\circ h\)=0,$ we have that $Ad=0$ and so $A(x+t\cdot d)=b$ for any $t\in \R.$ Because $x=y\circ y$ and $h\circ h\geq 0,$ it is easy to see that when $t>0$ is sufficiently small, $x+t\cdot d\in \R^n_+.$ Therefore, $d$ is a descent direction. 
\end{proof}

} 

{
\begin{remark}\label{donotuse}
In practice, we can use the method mentioned in this Subsection~\ref{subsec-verify} to recover the Lagrangian multiplier and check the KKT conditions. In the next section, we will design a hybrid algorithm combining Riemannian optimization and projected gradient method where the latter can also be used to recover the Lagrangian multiplier. Although the projected gradient step requires solving a larger problem ($m$ dimentional) compared to (\ref{drecov}), the Riemannian gradient descent steps ensure that we won't compute the projected gradient step too frequently. We will elaborate this in detail in the next section. In our implementation, we don't actually use the method in the current subsection to verify the optimality conditions because the Lagrangian multiplier coming from projected gradient step can already {be} used to verify the optimality conditions. {One} interesting future direction is to design a Riemannian optimization method without projected gradient step, where the method in Subsection~\ref{subsec-verify} is applied.
\end{remark}
}

This section has delved into the theoretical analysis of problem (\ref{LCPH}), covering both the conditions for optimality and the methods for their verification. In the subsequent section, we will study the manifold decomposition of $\PA$ which is the basis for our later algorithmic design.

\section{Manifold decomposition of $\PA$}\label{sec-manidcmp}

\subsection{Manifold decomposition based on regular representation}\label{subsec-regdecomp}

We define the following {equivalence} {relation} on $\PA.$
\begin{definition}\label{equidef}
Two points $y_1,y_2\in \PA$ are considered equivalent, denoted as $y_1\sim y_2,$ if they share the same regular representation. For any $y\in \PA,$ define the {equivalence} class $\M(y):=\left\{z\in \PA:\ z\sim y\right\}.$
\end{definition}

We have the following result about the decomposition of $\PA.$

\begin{proposition}\label{equitheo}
There are finitely many  {equivalence} classes $\M(y)$ in $\PA,$ each of which is a Riemannian manifold in $\R^n$ adopting the Euclidean metric $\<\cdot,\cdot\>$ with the following formula
\begin{equation}\label{submani}
\M(y)=\left\{ z\in \R^n:\ MA (z\circ z)=Mb,\ \textrm{rank}\(MA\Diag{z}\)=r,\ \supp(z)\subset \I \right\},
\end{equation}
where $\I:=\cup_{z\sim y} \supp(z).$ In addition, $\M(y)\cap \R^n_+$ is connected. 
\end{proposition}

\begin{proof}
Consider $y\in \PA,$ with regular representation $(M,N,r)$ and $x=y\circ y.$ From  Definition~\ref{Regrep}, $y_1\sim y_2$ if $\supp(y_1)=\supp(y_2).$ This implies that there are at most $2^n$ {equivalence} classes in $\PA.$ Then we have the following result:
\begin{equation}\label{submani0}
\M(y)\subset \left\{ z\in \R^n:\ MA (z\circ z)=Mb,\ {\rm rank}\(MA\Diag{z}\)=r,\ \supp(z)\subset \I \right\}.
\end{equation}
Moreover, for any $w\in {\rm conv}\left\{ z\circ z:\ z\in \M(y) \right\},$ from
{the} linearity of the operator $x\rightarrow A\Diag{x},$ we have that $N A \Diag{w}=0$ and $\rr(A\Diag{w})\geq r.$ Thus, $\sqrt{w}$ also has the regular representation $(M,N,r).$ Because there exists $w\in {\rm conv}\left\{ z\circ z:\ z\in \M(y) \right\}$ such that $\supp(w)=\I,$ we have that for any $z\in \PA$ such that $\supp(z)\subset \I,$ $N A\Diag{z}=0.$ Therefore, we have the following result
\begin{equation}\label{submani1}
\left\{ z\in \R^n:\ MA (z\circ z)=Mb,\ {\rm rank}\(MA\Diag{z}\)=r,\ \supp(z)\subset \I \right\}\subset \M(y).
\end{equation}
Combining (\ref{submani0}) and (\ref{submani1}), we have that (\ref{submani}) holds. From (\ref{submani}), $\M(y)$ is the Cartesian product of a submanifold in $\R^\I$ satisfying the LICQ property {everywhere}\footnote{{The LICQ property implies that the neighbourhood of the point is diffeomorphic to an open set in the Euclidean space through the implicit function theorem.}} and the zero-degree manifold $\{ \vec{0} \}$ in $\R^{[n]\setminus \I}.$ Therefore, $\M(y)$ is a Riemannian manifold in $\R^n.$ 

The remaining task is to show that $\M(y)\cap \R^n_+$ is connected. For any $y_1,y_2\in \M(y)\cap \R^n_+,$ from the above analysis, we have that for any $t\in[0,1],$ $\sqrt{ty_1\circ y_1+(1-t)y_2\circ y_2}\in \M(y)\cap \R^n_+.$ Because this continuous curve connects $y_1$ and $y_2,$ $\M(y)\cap \R^n_+$ is connected.
\end{proof}

\begin{remark}\label{remconnec}
Proposition~\ref{equitheo} implies that the intersection of each submanifold and 
$\R^n_+$ is a connected set. Because the signs of each element of $y$ have no influence of the feasibility and function value of (\ref{LCPH}), {the potential disconnectedness of each submanifold cause no issue to our Riemannian optimization algorithm.}
\end{remark}

\subsection{Some examples}\label{subsec-examples}
Since our manifold decomposition is abstract, {we will} use two examples to illustrate it. 

\medskip

\noindent{\bf Simplex.}
Perhaps the simplest example is the simplex $\C_{A,b}:=\left\{x\in \R^n_+:\ {\bf 1}_n^\top x=1\right\}.$ In this case, $\PA:=\left\{y\in \R^n:\ y^\top y=1\right\}$ is a unit sphere that is widely used in Riemannian optimization. {The  simplex to sphere parametrization has been used more than two decades ago in \cite[chapter 4]{helmke2012optimization}.} For any $y\in \PA,$ because it is regular, it has a consistent regular representation $(1,[\;],1).$ Thus, our manifold decomposition only 
consists of a single manifold, which is $\PA$ itself. Actually, we can also decompose $\PA$ according to the support of its elements, which would result in an explicit manifold decomposition. However, if we decompose $\PA$ in this way, the number of submanifolds will be exponentially large because there are in total $2^n-1$ support sets in $\PA.$ Therefore, our decomposition in Definition~\ref{equidef} avoids unnecessary separation of a manifold.

\medskip

\noindent{\bf Birkhoff polytope.} 
Let's consider a more complicated example, which is the following Birkhoff polytope \cite{birkhoff1946three}:
\begin{equation}\label{DST}
\B_n:=\left\{X\in \R^{n\times n}_+:\ X{\bf 1}_n = {\bf 1}_n,\ X^\top {\bf 1}_n = {\bf 1}_n \right\}.
\end{equation}
The Hadamard-parametrized polyhedron of $\B_n$ is as follows:
\begin{equation}\label{DST}
\BH_n:=\left\{Y\in \R^{n\times n}:\ \dd(YY^\top) = {\bf 1}_n,\ \dd(Y^\top Y)_{1:n-1} = {\bf 1}_{n-1} \right\},
\end{equation}
where we have removed one redundant constraint to ensure Assumption~\ref{ass1}. Now we move on to study its manifold decomposition.
For any $Y\in \BH_n,$ we can use a bipartite graph to represent its sparsity pattern.
\begin{definition}\label{bipardefi}
For any $Y\in \BH_n,$ define a bipartite graph $\BG(Y):=\([n],[n],\E\)$ such that for any $(i,j)\in [n]\times [n], $ $(i,j)\in \E$ if and only if $Y_{ij}\neq 0.$ The underlying sparsity pattern of $Y$ is the bi-adjacency matrix of $\BG(Y).$
\end{definition}
With the above definition, we have the following result connecting the smoothness of $Y$ and the connectedness of $\BG(Y).$ We put its proof in the appendix.

\begin{proposition}\label{biparprop}
For any $Y\in \BH_n,$ it is smooth if and only if $\BG(Y)$ is connected.
\end{proposition}

Proposition~\ref{biparprop} tells us that determining the singularity of a given point $Y\in \BH_n$ only requires checking the connectedness of its underlying sparsity pattern. Note that this can be done in $\O(n+|\supp(Y)|)$  time by Depth-First Search or Breadth-First Search. This is much cheaper than applying the LDL decomposition to the matrix $A\DD(x)A^\top,$ which has
the complexity of $\O\(n^3\).$ The following proposition constructs the manifold decomposition of $\BH_n$ explicitly. We also put its proof in the appendix.

\begin{proposition}\label{manidecB}
Consider $Y\in \BH_n$ such that $\BG(Y)$ has connected components $\sqcup_{i=1}^k \H_i$ for any $i\in [k]$ such that the vertex set $\V\( \H_i \)=\(\I_i,\J_i\)$ and $[n]=\sqcup_{i=1}^k \I_i=\sqcup_{i=1}^k \J_i.$ Then for any $Z\in \BH_n,$ $Z\in \M(Y)$ if and only if $\BG(Z)$ has the same connected components as $\BG(Y).$ 
\end{proposition}

For examples like the simplex and Birkhoff polytope, we are able to construct the manifold decomposition explicitly due to their special structures. However, this analysis is difficult to be extended to a general convex polyhedron. In the next section, we will introduce a globally convergent Riemannian optimization method for solving the general problem (\ref{LCP}). Specifically, details like the regular representation $(M,N,r)$, the maximal support $\I$ of $\M(y)$ as in (\ref{submani}), and the total number of submanifolds are not explicitly required for our algorithm's operations.

\section{Algorithm}\label{Sec-alg}
In this section, we will consider finding a first-order stationary point of (\ref{LCP}) by solving (\ref{LCPH}) as an auxiliary problem. Our strategy is to apply manifold decomposition to $\PA,$ mentioned in the previous section and employ Riemannian optimization techniques on the submanifolds. Detailed algorithmic design and convergence analysis will be presented in the subsequent subsections. Throughout this section, we let $(M,N,r)$ be a regular representation of $y$ and $x=y\circ y.$
\subsection{Riemannian gradient descent (RGD)}
Suppose we have some iterate $y\in \PA,$ which is also on the submanifold $\M(y)$ whose explicit formula might be unknown. When applying Riemannian gradient descent, we need {to compute the Riemannian gradient and retraction. Here we introduce some basic material on Riemannian optimization.} For more detailed information about Riemannian optimization, please refer to the textbooks \cite{intromani,absil2008optimization}.

\bigskip

\noindent{\bf Tangent space.} Consider $y\in \PA$ with regular representation $(M,N,r),$ $x=y\circ y$ and $\I:=\cup_{z\in \M(y)} \supp(z)$ to be the maximal support. Because of the LICQ property, the tangent space of $\M(y)$ at $y$ is given as follows:
\begin{equation}\label{tansp}
{\rm T}_y \M(y)=\left\{ h\in \R^n:\ MA (y\circ h)=0,\ \supp(h)\subset \I \right\}.
\end{equation}
{The collection of all the tangent spaces $\cup_{z\in \M(y)} \{(z,h) \,:\, h\in {\rm T}_z \M(y)\}$ is known as the tangent bundle and is denoted as ${\rm T}\M(y).$}
With the manifold decomposition structure of $\PA,$ we will discuss our Riemannian optimization approach in the following subsections.

\medskip

\noindent{\bf Riemannian gradient.} Because $\supp\(2\nabla \phi(x)\circ y\)\subset \supp(y)\subset \I$ and $\Ker[MA\Diag{y}]=\Ker[A\Diag{y}],$ we get the following formula:
\begin{eqnarray}
\g_{\M(y)} f(y) & :=& \P_{{\rm T}_y \M(y)}\( 2\nabla \phi(x)\circ y \)=\P_{\Ker[A\Diag{y}]}\( 2\nabla \phi(x)\circ y \)
\nonumber\\
&=& 2\(\nabla \phi(x)\circ y-\Diag{y}A^\top \lambda\),
\label{projeq}
\end{eqnarray}
where $\lambda$ is any solution of the linear system (\ref{linsys}) with one particular choice to be (\ref{linsoln}). From (\ref{projeq}) and Lemma~\ref{1stlem}, we have the following proposition.

\begin{proposition}\label{rgrad}
For any $y\in \PA,$ it is a first-order stationary point if and only if $\g f(y)=0,$ which is further equivalent to the complementarity with the multiplier $\lambda$ coming from the linear system (\ref{linsys}).
\end{proposition}

\noindent{\bf Retraction.} Apart from the Riemannian gradient provided in (\ref{projeq}), we also need a retraction mapping to ensure that the next iterate is also on the current manifold (see 3.6 of \cite{intromani}). A retraction ${\Re:} {{\rm T} \M(y)}\rightarrow \M(y)$ is a smooth mapping such that for any $z\in \M(y)$ and $h\in {\rm T}_z \M(y)$, it satisfies
\begin{itemize}
\item $\Re(z,0)=z$, ${\rm D}\Re(z,0)[h]=h$ for any $h\in {\rm T}_z \M(y).$
\end{itemize}
Notably, a retraction mapping is not unique. For specific manifolds such as the unit sphere and Stiefel manifold \cite[chapter 7]{intromani}, there {exist} retraction mappings with analytical formulas. For a general manifold defined from equality constraints with LICQ property, we can use the Newton retraction proposed by Zhang in \cite{zhang2020newton} to compute the retraction iteratively. In detail, the Newton retraction for {$\Re(y,-t\cdot \g f(y))$} is the Gauss-Newton method applied to solve the nonlinear system $M A (z\circ z)=Mb$ with the initial point set to be 
{$y-t\, \g f(y).$} Suppose $z_k$ is the $k$th iteration of the Gauss-Newton method, then 
$z_{k+1}$ is updated as follows: 
\begin{multline}\label{retraceq}
2MA \DD(z_k\circ z_k)A^\top M^\top \lambda_k=Mb-MA (z_k\circ z_k),\ z_{k+1}=z_k+\DD(z_k)A^\top M^\top \lambda_k.
\end{multline}
If $\supp(z_k)\subset \supp(y),$ we also have that $\supp(z_{k+1})\subset \supp(y).$ From (\ref{projeq}), we have that $\supp\( y+t\cdot \g f(y) \)\subset \supp(y).$ Therefore, the Gauss-Newton method preserves the zero entries of $y.$ 

Because $\supp(z_k)\subset \supp(y),$ we have that $N A\DD(z_k)=0$ and $Nb=0.$ Similar to the  (\ref{linsys}), by using the LDL decomposition, $A\DD(z_k\circ z_k)A^\top=L_k\DD(d_k)L_k^\top,$ the Newton retraction iteration can be equivalently written as 
\begin{equation}\label{retraceq1}
2L_k\DD(d_k)L_k \lambda_k=b-A (z_k\circ z_k),\ z_{k+1}=z_k+\DD(z_k)A^\top  \lambda_k,
\end{equation}
where one particular choice of $\lambda_k$ is $L_k^{-\top }\DD(d_k)^{+}L_k^{-1}\(b-A (z_k\circ z_k)\)/2.$

\begin{remark}\label{remretrac}
Equations (\ref{projeq}) and (\ref{retraceq1}) indicate that we can calculate the projection and retraction mappings without the knowledge of the maximal support $\I$ or the regular representation $(M,N,r)$. This representation is concealed within the LDL decomposition of $A\DD(z_k\circ z_k)A^\top$. Importantly, when the stepsize is sufficiently small, the preservation of nonzero entries ensures that the new point remains on the current submanifold.

It is crucial to note that Newton retraction is conceptually the limit point of the Gauss-Newton method. Attaining this limit point exactly in practice is generally not feasible due to the presence of a nonzero residual. However, given Gauss-Newton method's local quadratic convergence and the proximity of points on the tangent space to the manifold before retraction, the residual of the Newton retraction can typically reach the machine precision in just a few iterations. This means that we can practically consider the iterate as the limit point. In our tests, even setting the residual tolerance of the retraction to values like $10^{-8}$, which is far from the machine precision, did not pose any numerical issues for our algorithm.

While we acknowledge there is a slight discrepancy between the theory developed and practical implementation due to our use of the (inexact) Newton retraction, we believe it might be overly cautious to limit ourselves to straightforward manifolds such as spheres or orthogonal matrices with explicit retractions. Exploring the convergence behavior of Riemannian optimization methods with inexact retraction mappings is a promising avenue for future research.
\end{remark}

With the Riemannian gradient and retraction mapping, we state and prove the following proposition on the sufficient decrease of function value.

\begin{proposition}\label{gdprop}
Let $\delta\in (0,1)$ be a given constant.
For any $y\in \PA$ such that $\g f(y)\neq 0,$ there exists $\beta_{y,1},\ c_{y,1}>0$ such that for any $z\in \M(y)\cap B_{\beta_{y,1}}(y)$,
\begin{equation}\label{grades}
{\rm RGD:}\ f\( \Re\(z,-t \g f(z)\) \)\leq f(z)-t\cdot\delta\| \g f(z) \|^2,\ \forall t\in [0,c_{y,1}].
\end{equation}
\end{proposition}

\begin{remark}\label{Rline-search}
In Proposition~\ref{gdprop}, $c_{y,1}$ might be unknown. In practice, we can use the following Armijo linesearch with backtracking rate of $0.5$ \cite[section 3]{numerical} to adaptively find a step size:
\begin{equation}\label{linsearch1}
\min\left\{k\in \mathbb{N}:\ f\( \Re\(y,-(1/2)^k \g f(y)\) \)\leq f(y)-(1/2)^k\cdot\delta\| \g f(y) \|^2\right\}.
\end{equation}
Proposition~\ref{gdprop} ensures that the stepsize $(1/2)^k\geq \min\{1,c_{y,1}/2\}.$
\end{remark}

Proposition~\ref{rgrad} and Proposition~\ref{gdprop} show that we can use RGD to decrease the function value if the complementarity is violated. However, the first-order stationarity of (\ref{LCPH}) is not enough to ensure $x=y\circ y$ to be a first-order stationary point of (\ref{LCP}) because we have to further ensure the dual feasibility. A simple example is that for a convex function $\phi$, a first-order stationary point of (\ref{LCP}) is also a global minimizer, but a first-order stationary point of (\ref{LCPH}) can be a global maximizer. 
Besides, a sufficient decrease is not assured when $\| \g f(y) \|$ is large because $\beta_{y,1}$ and $c_{y,1}$ may not have a uniform lower bound. 
The aforementioned inconsistency is caused by the possible unboundedness of $f(y)$ and $\M(y)$'s curvature, with the latter issue being non-negligible due to the potential openness of the submanifold $\M(y)$, as also seen in fixed-rank manifolds \cite[section 7.5]{intromani}. In the next subsection, we will introduce the projected gradient method to tackle these issues.


\subsection{Projected gradient descent (PGD)}
As mentioned in the previous subsection, RGD may suffer from small stepsize issue because of the singularity of $\PA.$ One simple way to overcome this difficulty is to combine RGD with the projected gradient method \cite{birgin2000nonmonotone}, which updates the iteration as follows:
\begin{equation}\label{PG}
x_{k+1}:=\P_{\C_{A,b}}(x_k-t_k \nabla \phi(x_k)).
\end{equation}
This amounts to solving the following convex projection problem:
\begin{equation}\label{ProjPG}
\min\left\{\frac{1}{2}\|x-\( x_k-t_k \nabla \phi(x_k) \)\|^2 \;:\ Ax=b,\ x\in \R^n_+ \right\},
\end{equation}
whose dual problem is given by
\begin{equation}\label{ProjPGD}
\min_{\lambda\in \R^m}\Big\{ \frac{1}{2} \|\P_{\R^n_+}\big(x_k-t_k \nabla \phi(x_k)+A^\top \lambda\big)\|^2
-\<\lambda,b\> \Big\}.
\end{equation}
The latter can be solved by a (regularized) semi-smooth Newton method (see algorithm SSNCG1 of \cite{li2020efficient}). We need the following proposition, which originates from \cite{birgin2000nonmonotone}.

\begin{proposition}(Lemma 2.1 and Lemma 2.2 of \cite{birgin2000nonmonotone})\label{PGdefi}
For any $t>0$ and $x\in \C_{A,b},$ define $g_t(x):=\P_{\C_{A,b}}\( x-t \nabla \phi(x) \)-x$ and $h_t(x)=\| g_t(x) \|/t.$ We have that (1) $\< \nabla \phi(x),g_t(x) \>\leq -\| g_t(x) \|^2/t$; (2) $g_t(x)=0$ if and only if $x$ is a first-order stationary point of (\ref{LCP}); (3) $h_t(x)$ is monotonically nonincreasing with respect to $t.$
\end{proposition}

With Proposition~\ref{PGdefi}, we have the following result on the sufficient decrease property of the projected gradient method.
\begin{proposition}\label{PGprop}
Let $\delta\in (0,1)$ be a given constant. For any $x=y\circ y\in \C_{A,b}$ that is not a stationary point of (\ref{LCP}), there exists $\beta_{y,2},c_{y,2}>0$ such that for any $u\in \C_{A,b}\cap B_{\beta_{y,2}}(x)$,
\begin{equation}\label{PGdes}
{\rm PGD:}\ \phi\( \P_{\C_{A,b}}\( u-t \nabla \phi(u) \) \)\leq \phi(u)-t\cdot\delta h_t(u)^2
\quad \forall\; t\in (0,c_{y,2}].
\end{equation}
\end{proposition}

\begin{remark}\label{Pline-search}
In Proposition~\ref{PGprop}, $c_{y,2}$ is unknown in practice. Similar to the RGD method, we can perform the following Armijo linesearch to the PGD method:
\begin{equation}\label{linsearch2}
\min\left\{k\in \mathbb{N}^+:\ \phi\( \P_{\C_{A,b}}\( x-(1/2)^k \nabla \phi(x) \) \)\leq \phi(x)-(1/2)^k\cdot\delta h_{(1/2)^k}(x)^2\right\}.
\end{equation}
Proposition~\ref{PGprop} ensures that the step size $(1/2)^k\geq \min\{1,c_{y,2}/2\}.$
\end{remark}

Although PGD has a nice sufficient decrease property, solving problem (\ref{ProjPGD}) might be expensive especially when (\ref{LCP}) is primal degenerate. In the next subsection, we will design a hybrid method that uses the RGD whenever the iterate is not close to a singular 
point of $\PA$ and employs the PGD otherwise. As the set of singular points of $\PA$ is ``sparse''
in $\PA$, we expect that the more expensive PGD step will be used infrequently.


\subsection{A hybrid method with convergence guarantee}
In this subsection, we will design our algorithm, which is a combination of the RGD and PGD methods. The main idea is to use the former  to decrease the function value when the iteration 
is on the smooth part of $\PA$ and switch to the projected gradient method when it becomes ``nearly singular". We define the near-singularity as follows.

\begin{definition}\label{defi-sing}
Consider $y\in \M(y)$ with regular representation $(M,N,r)$ and $x=y\circ y.$ Let $L\DD(d)L^\top$ be the LDL decomposition of $A\DD(x)A^\top.$ Suppose $d$ is sorted as $d_{i_1}\geq d_{i_2}\geq \cdots\geq d_{i_r}>0=d_{i_{r+1}}=\cdots=d_{i_n}.$ For any $\sigma>0,$ we say that $y$ is $\sigma-$singular if $d_{i_r}<\sigma.$ Otherwise, we say that $y$ is $\sigma-$regular, and define $\M_\sigma(y):=\left\{z\in \M(y):\ z{\rm\  is\ } 
\mbox{$\sigma-$regular}\right\}.$
\end{definition}

From Lemma~\ref{LDL}, we know that the vector $d$ is unique in the LDL decomposition of $A\DD(x)A^\top.$ Thus, the above definition is only a property about $y$ and it does not depend on the manifold decomposition of $\PA$.
 Now with Definition~\ref{defi-sing}, we can state our algorithm as follows.

\begin{description}
\item [Algorithm 1:] Choose $y_0\in \PA,$  and $\delta,\epsilon,\sigma\in (0,1).$ Set $x_0=y_0\circ y_0$ and $k=0.$
\item[Step 1.] 
{(a)} If $y_k$ is $\sigma-$regular and $\| \g f(y_k) \|>\epsilon$, then $y^+\leftarrow\Re_{y_k}\(y_k,-t_k \g f(y_k)\)$ satisfies (\ref{grades}) for some $t_k$ coming from the linesearch procedure (\ref{linsearch1}). Set $x^+\leftarrow y^+\circ y^+.$ Goto Step 2.

{(b)} If $y_k$ is $\sigma-$singular or $\| \g f(y_k) \|\leq \epsilon$, then $x^+\leftarrow \P_{\C_{A,b}}(x_k-t_k\nabla \phi(x_k))$ satisfies (\ref{PGdes}) for some $t_k$ coming from
the linesearch procedure (\ref{linsearch2}). Set $y^+\leftarrow \sqrt{x^+}.$ Goto Step 2.
\item[Step 2.] Set $x_{k+1}\leftarrow x^+,$ $y_{k+1}\leftarrow y^+,$ $k\leftarrow k+1.$ Goto Step 1.
\end{description}

In Algorithm 1, we apply RGD when the current iteration is $\sigma-$regular and the Riemannian gradient norm is larger than some threshold $\epsilon.$ These two conditions are important for the good behaviour of RGD. If any one of the two conditions is not satisfied, we use the PGD method. 
In practice, we choose $\epsilon$ and $\sigma$ to be $10^{-10}$ to avoid using the PGD method
 too frequently. Another point to note is that the RGD step keeps the iteration on the current submanifold. However, the PGD step may switch the iteration point to another submanifold. 
Thus the PGD step allows us to adaptively identify a submanifold of $\PA$ while guaranteeing the decrease in the function value. {Because Gauss-Newton method only has local convergence guarantee, we will first randomly generate a point in $\R^n$ and project it onto the convex polyhedron $\mathcal{C}_{A,b}$ to get $x_0.$ Such step can be computed by the semi-smooth Newton method or the code \texttt{PPROJ} \cite{hager2016projection}. Then we factorize $x_0=y_0\circ y_0$ to get the initial point $y_0\in \PA.$}

In Algorithm 1, we intentionally use a simple algorithmic framework to facilitate our convergence analysis. In order to achieve good practical performance, it is important to use some numerical techniques to accelerate our algorithm. For example, when choosing the stepsize, compared with the naive Armijo linesearch mentioned in Remark~\ref{Rline-search} and Remark~\ref{Pline-search}, it would be better to use the Barzilai-Borwein stepsize together with a non-monotone linesearch \cite{birgin2000nonmonotone}.

Also, {in order to terminate our algorithm, we will} check the KKT conditions of the current iteration and terminate our algorithm when the following KKT residuals are smaller than some tolerance:
\begin{eqnarray}\label{KKTresidual}
&&{\rm Rp}:=\max\left\{\frac{\|Ax-b\|}{1+\|b\|},\frac{\|\P_{\R^n_-}(x)\|}{1+\|x\|}\right\},\ {\rm Rd}:=\frac{\|\P_{\R^n_-}\(\nabla \phi(x)-A^\top \lambda\)\|}{1+\| \nabla \phi(x) \|},
\nonumber \\ 
&& {\rm Rc}:=\frac{|\< \nabla \phi(x)-A^\top \lambda,x \>|}{1+\| \nabla \phi(x) \|}.
\end{eqnarray}

{To check the above KKT conditions, we have to compute the Lagrangian multiplier $\lambda.$ One choice is the method discussed in Subsection~\ref{subsec-verify}. However, through our numerical experiments, we find that it is more convenient to use the multiplier computed from the projected gradient step as the Lagrangian multiplier. In our implementation, we 
compute the multiplier to check the KKT residuals based on  the following two cases:
\begin{description}
\item [Case 1.] If the current iteration is smooth, the unique Lagrangian multiplier $\lambda$ will be obtained from the linear system of the projection mapping (\ref{eqproj}). 
\item [Case 2.] If the current iteration is singular, the Lagrangian multiplier $\lambda$ will be obtained from the optimal solution of the projected gradient subproblem (\ref{ProjPGD}).
\end{description}
}

{The above Case 2 is essentially using projected gradient method to recover the dual variable. Since projected gradient method is not used in every iteration, we will periodically switch to projected gradient, say once every 20 iterations for degenerate problems, to check KKT conditions so that our algorithm can find a solution of required accuracy in time.}

Now we state and prove the convergence result of Algorithm 1.

\begin{theorem}\label{conv}
For any $\delta,\sigma,\epsilon\in (0,1)$ and $y_0\in \PA.$ Suppose $\{y_k\}_{k\in \mathbb{N}^+}$ is the sequence generated from Algorithm 1 and $x_k=y_k\circ y_k.$ Then for any accumulation point $y^*$ of $\{y_k\}_{k\in \mathbb{N}^+}$,  $x^*:=y^*\circ y^*$ is a first-order stationary point of (\ref{LCP}).
\end{theorem}

\begin{proof}
Assume on the contrary that $x^*$ is not a first-order stationary point of (\ref{LCP}). 
Then, according to Proposition~\ref{PGdefi}, $h_1(x^*) = \norm{ g_1(x^*)} > 0$.
Decompose $\{i_k:\ k\in \mathbb{N}\}=\I\sqcup \J$ such that $\I$ corresponds to the iterations for RGD
 and $\J$ corresponds to the iterations for PGD. 
We have two cases to consider.

\medskip

\noindent{\bf Case 1:} $|\J|=\infty$.\quad
From Proposition~\ref{PGprop} and the linesearch strategy mentioned in Remark~\ref{Pline-search}, we have that there exists $N\in \mathbb{N}^+$ and a positive constant $\alpha\in (0,1)$ such that for any $k\in \J$ and $k>N,$ the stepsize $t_k\in (\alpha,1]$ and $h_{t_k}(x_k)\geq h_{1}(x_k)\geq h_{1}(x^*)/2.$ Note that we have used the fact that $h_t(x)$ is monotonically nonincreasing from Proposition~\ref{PGdefi} and the continuity of $h_1(x).$ From (\ref{PGdes}), we have that 
\begin{equation}\label{Sep_24_5}
\phi(x_{k+1})\leq \phi(x_k)-\delta \alpha h_{1}(x^*)^2/2.
\end{equation}
From (\ref{Sep_24_5}) and the monotonically non-increasing property of Algorithm 1, we have that $\phi(x^*)={\inf}\left\{ \phi(x_k):\ k\in \J \right\}=-\infty,$ which contradicts the fact that $\phi(x^*)\in \R.$

\bigskip
\noindent{\bf Case 2:} $|\I|=\infty.$\quad
From Proposition~\ref{equitheo}, there are only finitely many submanifolds in the manifold decomposition of $\PA.$ Thus, there exists $\I_1\subset \I$ such that for any $k\in \I_1,$ $y_k\in \M(y_l),$ for some $l\in \I_1.$ From Step 1 in Algorithm 1, we have that for any $k\in \I_1,$ $\|\g f(y_k)\|>\epsilon$ and $y_k\in \M_\sigma(y_l)$. From Lemma~\ref{sigmaregular}, we have that $y^*\in \M_\sigma(y_l).$ From Proposition~\ref{gdprop}, we have that there exists $N\in \mathbb{N}^+$ such that for any $k\geq N$ and $k\in \I_1$, the stepsize obtained by the linesearch strategy in Remark~\ref{Rline-search} satisfies $t_k\in (\alpha,1]$ for some 
positive constant $\alpha.$ From (\ref{grades}), we have that
\begin{equation}\label{Sep_24_6}
f(y_{k+1})\leq f(y_k)- \alpha\delta \epsilon^2.
\end{equation}
Similar to Case 1, we obtain $\phi(x^*)=-\infty,$ which contradicts the fact that $\phi(x^*)\in \R.$
\end{proof}

\begin{remark}
In this section, we have stated our hybrid method together with its convergence analysis. Actually, from our convergence analysis, our algorithm still converges if we use the PGD step in Step 1(a) of Algorithm 1. Our tests have shown that the RGD step always preserves the support of the iterate, that is, some entries will approach zero without reaching it exactly. On the other hand, the PGD step can reduce these values to exactly zeros, thereby modifying the iterate's support. To enhance the vector's sparsity early on, we switch to the PGD step occasionally, perhaps every ten steps even in Step 1(a). This technique is particularly useful for problems where most of the entries at the optimal solution are zeros. 
\end{remark}

One might naturally question whether our algorithm offers any benefit over purely using
projected gradient method. Although it's not yet proven theoretically, our numerical tests indicate that the Gauss-Newton method employed in the retraction mapping 
converges faster than the semi-smooth Newton method 
for solving (\ref{ProjPGD}), especially for degenerate problems. One potential reason is that the point on the tangent space before retraction is usually very close to the manifold especially when the algorithm is about to converge and the stepsize becomes small. In this case, the initial point is inside the quadratically convergent region and the residual is already very small. So after one or two iterations, the Gauss-Newton method's residual will reach machine precision. On the contrary, for the projected gradient method, the semismooth Newton method employed to solve the  subproblem \eqref{ProjPGD} generally does not enjoy the quadratic convergence property in the early iterations, and worse still, there is no quadratic convergence when the subproblem
is degenerate.

Despite the advantages of RGD, we do not claim that Algorithm 1 will consistently outperforms the projected gradient method. 
Our goal is to present it as a viable alternative, or as an acceleration, to the well established projected gradient method. Ideally, our algorithm should exhibit superior performance in some interesting applications. In the following section, we will conduct numerical experiments 
to compare the performance our approach with the projected gradient method.



\section{Numerical experiments}\label{Sec-numer}

In this section, we conduct numerical experiments to verify the effectiveness of Algorithm 1. As mentioned in the previous section, our algorithm is closely related to the projected gradient method so we will compare Algorithm 1 with it. 
Instead of implementing the projected gradient method ourselves, we compare with the well-developed solver \texttt{PASA} in the package \texttt{SuiteOPT}\footnote{Source codes from \href{https://people.clas.ufl.edu/hager/software/}{https://people.clas.ufl.edu/hager/software/}.} by Hager and Zhang \cite{hager2023algorithm}, which 
 is a software designed to solve optimization problems over a convex polyhedral set. 
According to \cite{hager2023algorithm}, \texttt{PASA} is the product of over 
20 years of cumulative research effort.
It combines a non-monotone projected gradient method and an active set algorithm that explores the faces of the convex polyhedron \cite{di2023stationarity}. We use the KKT residual in (\ref{KKTresidual}) to measure the level of first-order stationarity of their outputs. We set the maximum running time to be 36000 seconds and the tolerance ${\rm tol}=10^{-6}.$ {When solving the linear systems in the projection and retraction steps, we will first use the preconditioned conjugate gradient (PCG) method {with diagonal preconditioing and only switch to using the LDL decomposition of the coefficient matrix as the preconditioner when the PGG method cannot reach the tolerance of ${\rm tol}/10^3$ within say, 20 iterations. We will continue to use the old LDL decomposition as the preconditioner for solving the subsequent linear systems until the PCG method 
triggers another need to recompute the LDL decomposition.}
This strategy has been explored in our previous feasible method for solving general low-rank SDP problems \cite{tang2024feasible}.} All the experiments are run using Matlab R2021b on a Workstation with a Intel(R) Xeon(R) CPU E5-2680 v3 @ 2.50GHz Processor and 128GB RAM.

\subsection{Convex quadratic programming}
Consider the following standard form convex quadratic programming problem (CQP)
\begin{equation}\label{CQP}
\min\left\{ x^\top Qx/2+c^\top x:\ Ax=b,\ x\in \R^{n}_+ \right\},
\end{equation}
where $Q\in \S^n_+,$ $A\in \R^{m\times n},$ $b\in \R^m$ and $c\in \R^n$ are given data. We randomly generate (CQP) problems following the numerical tests in \cite{ding2023squared}. For completeness, we restate their generation procedure in our paper. The matrix Q is randomly generated with eigenvalues log-uniformly distributed in $[\kappa,1]$ for some $\kappa\in (0,1),$ and eigenvectors oriented randomly. The vector $c$ is set to be $-Qx^{(1)}$ such that $x^{(1)}$ has  i.i.d entries from normal distribution $N(0,1)$. The matrix $A$ and primal feasible point $x^{(2)}$  have entries drawn randomly  from the uniform distribution in $[0,1].$ The vector $b$ is set as $Ax^{(2)}.$ The {\sc Matlab} codes for the data generation is as follows:
\begin{center}
\texttt{ rng('default'); l=log(kappa); u=0; egv = exp(unifrnd(l,u,n,1)); S=randn(n,n); S = S+S'; [P,S] = eig(S); Q = (P.*egv')*P'; x1 = randn(n,1); c = -Q*xs; A = rand(m,n); x2 = rand(n,1); b=A*x2; }
\end{center}

Apart from \texttt{PASA}, we also compare with Gurobi version 10.0.3 \cite{gurobi}\footnote{\href{https://www.gurobi.com/downloads/gurobi-optimizer-release-notes-v10-0-3/}{https://www.gurobi.com/downloads/gurobi-optimizer-release-notes-v10-0-3/}.}. We use the following parameters for Gurobi:
\begin{center}
\texttt{ par.Presolve = 0; par.TimeLimit = 36000; par.OptimalityTol = 1e-6; }
\end{center}
{Note that we don't use the presolving procedure in Gurobi because we find that for our test problems, the presolving procedure usually increases the total running time of Gurobi slightly without making any simplification of the problem.}

{
\begin{center}
\begin{tiny}
\begin{longtable}{|c|c|cccccc|}
\caption{Comparison of Algorithm 1, \texttt{PASA} and Gurobi for random convex quadratic programming problems.}
\label{TCQP_random}
\\
\hline
problem & algorithm & Rp & Rd & Rc& obj & time & iter/ldl/pg \\ \hline
\endhead
n = 2000 & Algorithm 1 & 1.91e-16 & 9.99e-07 & 1.35e-09 & -1.5549714e+02& 1.80e+00 & 165/ 4/ 0 \\
m = 500 & Gurobi & 2.39e-13 & 0.00e+00 & 8.56e-09 & -1.5549714e+02& 6.00e+00 & \\
$\kappa=0.10$ & PASA & 3.77e-15 & 4.20e-07 & 1.45e-07 & -1.5549714e+02& 2.28e+00 & \\
\hline
n = 2000 & Algorithm 1 & 4.33e-11 & 8.91e-07 & 6.57e-10 & -1.2811790e+02& 2.46e+00 & 262/ 5/ 0 \\
m = 500 & Gurobi & 1.13e-13 & 0.00e+00 & 9.99e-09 & -1.2811790e+02& 6.87e+00 & \\
$\kappa=0.01$ & PASA & 3.69e-15 & 3.53e-07 & 2.17e-07 & -1.2811790e+02& 2.74e+00 & \\
\hline

n = 2000 & Algorithm 1 & 1.44e-12 & 8.40e-07 & 9.02e-09 & -4.8472049e+01& 1.27e+01 & 500/10/ 0  \\
m = 1000 & Gurobi & 3.46e-13 & 0.00e+00 & 7.87e-09 & -4.8472049e+01& 9.96e+00 & \\
$\kappa=0.10$ & PASA & 7.40e-15 & 7.58e-08 & 2.04e-08 & -4.8472049e+01& 7.18e+00 & \\
\hline
n = 2000 & Algorithm 1 & 8.28e-15 & 9.80e-07 & 5.89e-11 & -6.3029604e+01& 1.02e+01 & 176/11/ 0  \\
m = 1000 & Gurobi & 1.66e-13 & 0.00e+00 & 3.02e-09 & -6.3029604e+01& 1.01e+01 & \\
$\kappa=0.01$ & PASA & 2.14e-14 & 4.50e-07 & 3.44e-07 & -6.3029604e+01& 1.01e+01 & \\
\hline

n = 5000 & Algorithm 1 & 1.48e-15 & 9.64e-07 & 4.37e-09 & -3.6460458e+02& 1.56e+01 & 178/ 4/ 0  \\
m = 1250 & Gurobi & 1.53e-13 & 0.00e+00 & 2.34e-08 & -3.6460459e+02& 8.61e+01 & \\
$\kappa=0.10$ & PASA & 5.59e-15 & 1.05e-07 & 1.38e-08 & -3.6460459e+02& 3.20e+01 & \\
\hline

n = 5000 & Algorithm 1 & 2.29e-16 & 9.19e-07 & 5.29e-09 & -3.0196284e+02& 2.19e+01 & 255/ 5/ 0  \\
m = 1250 & Gurobi & 1.57e-13 & 0.00e+00 & 8.64e-08 & -3.0196284e+02& 8.66e+01 & \\
$\kappa=0.01$ & PASA & 3.76e-15 & 1.45e-07 & 2.38e-08 & -3.0196284e+02& 5.20e+01 & \\
\hline

n = 5000 & Algorithm 1 & 3.93e-15 & 7.99e-07 & 5.53e-07 & -1.3398205e+02& 4.50e+01 & 250/ 9/ 0  \\
m = 2500 & Gurobi & 9.88e-14 & 0.00e+00 & 2.55e-09 & -1.3398205e+02& 1.55e+02 & \\
$\kappa=0.10$ & PASA & 1.14e-14 & 5.35e-08 & 5.34e-08 & -1.3398205e+02& 1.08e+02 & \\
\hline

n = 5000 & Algorithm 1 & 1.09e-13 & 9.29e-07 & 5.22e-09 & -1.5828515e+02& 4.15e+01 & 214/10/ 0  \\
m = 2500 & Gurobi & 2.52e-13 & 0.00e+00 & 1.97e-08 & -1.5828515e+02& 1.43e+02 & \\
$\kappa=0.01$ & PASA & 1.22e-14 & 3.04e-07 & 1.24e-07 & -1.5828515e+02& 1.34e+02 & \\
\hline

n = 10000 & Algorithm 1 & 2.13e-14 & 3.58e-07 & 7.08e-10 & -7.6102357e+02& 5.61e+01 & 201/ 5/ 0 \\
m = 2500 & Gurobi & 5.48e-14 & 0.00e+00 & 1.36e-08 & -7.6102357e+02& 7.76e+02 &  \\
$\kappa=0.10$ & PASA & 1.27e-14 & 1.48e-07 & 7.59e-08 & -7.6102357e+02& 2.49e+02 &  \\
\hline

n = 10000 & Algorithm 1 & 6.96e-12 & 9.15e-07 & 7.94e-10 & -6.2019982e+02& 1.02e+02 & 306/ 5/ 0  \\
m = 2500 & Gurobi & 1.17e-13 & 0.00e+00 & 4.05e-08 & -6.2019982e+02& 7.49e+02 & \\
$\kappa=0.01$ & PASA & 2.29e-14 & 2.71e-07 & 2.36e-07 & -6.2019982e+02& 3.28e+02 & \\
\hline

n = 10000 & Algorithm 1 & 3.20e-16 & 2.96e-07 & 5.74e-10 & -2.9256608e+02& 2.33e+02 & 310/10/ 0 \\
m = 5000 & Gurobi & 2.31e-14 & 0.00e+00 & 8.82e-09 & -2.9256608e+02& 1.29e+03 & \\
$\kappa=0.10$ & PASA & 1.08e-14 & 2.85e-07 & 4.85e-07 & -2.9256608e+02& 7.79e+02 & \\
\hline

n = 10000 & Algorithm 1 & 1.12e-13 & 9.39e-07 & 7.17e-10 & -3.3885347e+02& 2.22e+02 & 295/ 9/ 0 \\
m = 5000 & Gurobi & 7.01e-14 & 1.19e-13 & 1.93e-08 & -3.3885347e+02& 1.26e+03 & \\
$\kappa=0.01$ & PASA & 1.73e-14 & 1.99e-07 & 1.10e-07 & -3.3885347e+02& 8.70e+02 & \\
\hline
\end{longtable}
\end{tiny}
\end{center}
}

From the data in Table~\ref{TCQP_random}, it is clear that all three methods can solve the problems to the required accuracy. Among them, Algorithm 1 is often the 
most efficient, particularly 
on larger problems. While the time saved isn't huge, it's still impressive, considering that Gurobi is a well-established commercial solver known for its quick solutions to convex quadratic programming problems. It's worth mentioning that we didn't put any restrictions on the number of threads Gurobi could use, which means that it could take full advantage of parallel computation with up to 24 threads on our workstation. 

Despite Algorithm 1 performing better than Gurobi in our tests, it would be premature to claim 
that it is a superior algorithm for all convex quadratic programming instances. The reason Algorithm 1 did so well in these instances is largely because the random matrix $Q$ is fully dense, which tends to slow down Gurobi's interior point method due to the time-consuming dense Cholesky factorization of an $n$ by $n$ matrix. Also, the problems we generated randomly don't usually have degeneracy issues, thus allowing Algorithm 1, which is a first-order method, to perform efficiently. {Our algorithm never switches to the projected gradient step because the optimal solution of a randomly generated problem is typically smooth and the diagonal elements in the LDL decomposition of \eqref{eqproj}
is bounded away from zero.} On the other hand, when dealing with problems where the matrices $Q$ and $A$ are highly sparse and have conducive structures, Gurobi is likely to far surpass our algorithm. We've included several examples from the Maros-M\'esz\'aros benchmark dataset \cite{maros1999repository} to illustrate this point in Table~\ref{TCQP_MM}.

{
\begin{center}
\begin{tiny}
\begin{longtable}{|c|c|cccccc|}
\caption{Comparison of Algorithm 1, \texttt{PASA} and Gurobi for selected problems from Maros-M\'esz\'aros benchmark dataset.}
\label{TCQP_MM}
\\
\hline
problem & algorithm & Rp & Rd & Rc& obj & time & iter/ldl/pg \\ \hline
\endhead
AUG2DCQP & Algorithm 1 & 5.16e-13 & 4.48e-14 & 4.38e-13 & 4.5649190e+04& 7.37e-01 & 10/11/ 1 \\
n=20200 & Gurobi & 3.86e-15 & 0.00e+00 & 1.84e-06 & 4.5649190e+04& 1.93e-01 & \\
m=10000 & PASA & 3.90e-15 & 2.42e-17 & 7.78e-07 & 4.5649190e+04& 3.75e-01 & \\
\hline
AUG2DQP & Algorithm 1 & 6.24e-15 & 8.25e-07 & 5.34e-09 & 4.4254158e+04& 1.65e+00 & 55/22/ 5 \\
n=20200 & Gurobi & 4.32e-15 & 0.00e+00 & 1.51e-06 & 4.4254158e+04& 1.81e-01 & \\
m=10000 & PASA & 8.22e-15 & 1.32e-08 & 5.84e-07 & 4.4254158e+04& 2.20e-01 & \\
\hline
BOYD1 & Algorithm 1 & 1.73e-14 & 9.99e-07 & 6.79e-07 & -6.3853592e+01& 1.05e+02 & 3265/410/66 \\
n=93261 & Gurobi & 2.05e-14 & 0.00e+00 & 7.60e-08 & -6.3853593e+01& 1.20e+00 & \\
m=18 & PASA & 1.89e-14 & 1.60e-07 & 3.86e-07 & -6.3853593e+01& 6.10e+00 & \\
\hline
CONT-100 & Algorithm 1 & 3.47e-13 & 1.33e-08 & 8.36e-08 & -4.6329622e+02& 1.55e+00 & 15/ 8/ 1 \\
n=10197 & Gurobi & 6.03e-14 & 0.00e+00 & 6.27e-09 & -4.6329622e+02& 1.96e-01 & \\
m=9801 & PASA & 7.04e-14 & 2.73e-09 & 1.67e-07 & -4.6329622e+02& 3.56e-01 & \\
\hline
CONT-101 & Algorithm 1 & 7.04e-14 & 6.27e-07 & 1.70e-07 & 3.7629251e+01& 3.22e+00 & 46/21/ 0 \\
n=9900 & Gurobi & 1.28e-10 & 0.00e+00 & 6.81e-08 & 3.7629251e+01& 2.18e-01 & \\
m=9801 & PASA & 2.30e-12 & 4.13e-08 & 4.62e-08 & 3.7629251e+01& 6.54e-01 & \\
\hline
CONT-200 & Algorithm 1 & 1.24e-12 & 2.21e-08 & 9.85e-07 & -9.1288126e+02& 1.27e+01 & 13/ 7/ 1 \\
n=40397 & Gurobi & 1.32e-13 & 0.00e+00 & 6.95e-08 & -9.1288126e+02& 7.90e-01 & \\
m=39601 & PASA & 2.80e-13 & 3.99e-07 & 1.26e-04 & -9.1288126e+02& 3.09e+00 & \\
\hline
CONT-201 & Algorithm 1 & 1.37e-13 & 2.06e-08 & 6.77e-07 & 7.3420291e+01& 4.92e+01 & 80/39/ 1 \\
n=39800 & Gurobi & 1.26e-09 & 0.00e+00 & 6.80e-08 & 7.3420291e+01& 1.03e+00 & \\
m=39601 & PASA & 3.64e-10 & 1.22e-07 & 5.28e-06 & 7.3420286e+01& 8.48e+00 & \\
\hline
\end{longtable}
\end{tiny}
\end{center}
}

From Table~\ref{TCQP_MM}, we can see that Gurobi performs significantly better than Algorithm 1, solving most of the problems within 1 second. This is {within} our expectation because the selected problems in the Maros-M\'esz\'aros benchmark dataset are very sparse and the matrices $Q$'s are diagonal. In this case, the per-iteration complexity of the interior point method is comparable to that of Algorithm 1. Remarkably, \texttt{PASA} also behaves quite well and is usually several times faster than Algorithm 1. One reason is that \texttt{PASA} makes use of low-rank update of Cholesky decomposition that also exploits the sparsity \cite{hager2016projection,davis2009dynamic}. Its efficiency in solving sparse convex quadratic problems has already been shown in the numerical results of \cite{hager2023algorithm}. Actually, as mentioned in \cite{hager2023algorithm}, the algorithms implemented in \texttt{PASA} have been developed for over 20 years.
Compared with them, our implementation of Algorithm 1 is still rather primitive and we believe that it has ample room for improvements
by using similar sparse matrix computation techniques
as those in \texttt{PASA}.

Our goal in this paper is not in developing a general solver that beats the existing ones. 
Rather, we are more interested to demonstrate that through Hadamard parametrization and the uncovered geometric structures, Riemannian based
optimization methods can be competitive in solving 
some specific non-convex programming problems over
a polyhedral set, which has recently become increasingly popular and found various applications. In the next subsection, we will discuss such an example.


\subsection{Gromov-Wasserstein learning}
In this paper, we are particularly interested in computing the following Gromov-Wasserstein distance (GWD)
\begin{equation}\label{GWDs}
{\rm GW}\(A,B,\mu,\nu\):=\min\Big\{ \frac{\mu^\top (A\circ A)\mu+\nu^\top (B\circ B)\nu}{2}-\<X,AXB\>:\ X\in \Pi(\mu,\nu) \Big\},
\end{equation}
where $A\in \S^m,$ $B\in \S^m$ are symmetric distance matrices, $\mu\in \R^m_+,$ $\nu\in \R^n_+$ are given probability vectors and the set $\Pi(\mu,\nu)$ is defined as follows:
\begin{equation}\label{Pi}
\Pi(\mu,\nu):= \left\{X\in \R^{m\times n}_+\,:\ X{\bf 1}_n=\mu,\ X^\top {\bf 1}_m=\nu\right\}.
\end{equation}
Problem (\ref{GWDs}) was first proposed by M{\'e}moli in \cite{memoli2011gromov} and has wide range of applications including graph partition and alignment \cite{xu2019gromov}, molecule analysis \cite{titouan2019sliced} and generative modeling \cite{bunne2019learning}. 

In our experiments, we consider the Gromov-Wasserstein distance in graph matching, where matrices $A$ and $B$ are adjacent matrices of some simple undirected graphs. We generate the graphs that imitate K-NN (K-nearest neighbors) graphs with certain randomness following the procedure mentioned in subsection 5.1 of \cite{xu2019gromov}. For completeness, we restate the data generation procedure. We generate the source graph $G(\V_s,\E_s)$ and target graph $G(\V_t,\E_t)$ as follows. For every $v\in \V_s,$ we randomly select $K\sim{\rm Poisson}\(0.1\times |\V_s|\)$ neighbors from $\V_s\setminus v,$ each of which has $w\sim{\rm Poisson}(10)$ interactions to $v.$ The target graph $G\(\V_t,\E_t\)$ is constructed by adding $q\%$ noise nodes and generating $q\%$ noisy edges via the simulation method. We choose $\mu={\bf 1}_m/m$ 
and $\nu={\bf 1}_{n}/n.$

{
\begin{center}
\begin{tiny}
\begin{longtable}{|c|c|cccccc|}
\caption{Comparison of Algorithm 1 and \texttt{PASA} for computing the Gromov-Wasserstein distance in graph matching problems. $m=|\V_s|,$ $p=q\%.$}
\label{TGWD}
\\
\hline
problem & algorithm & Rp & Rd & Rc& obj & time & iter/ldl/pg \\ \hline
\endhead
m = 100 & Algorithm 1 & 3.18e-11 & 6.79e-07 & 1.09e-09 & 7.2024465e-01 & 4.50e-01 & 61/ 7/ 9 \\
p = 1.00e-01 & PASA & 3.71e-15 & 7.25e-16 & 1.45e-18 & 7.2047603e-01& 3.48e-01 & \\
\hline
m = 100 & Algorithm 1 & 3.21e-11 & 4.75e-10 & 2.39e-12 & 1.6689367e+00 & 3.74e-01 & 72/29/12 \\
p = 3.00e-01 & PASA & 3.08e-15 & 2.28e-13 & 1.44e-15 & 1.6682743e+00& 2.13e-01 & \\
\hline
m = 500 & Algorithm 1 & 6.49e-11 & 8.12e-07 & 3.10e-10 & 7.7489653e-01 & 3.29e+00 & 59/19/12 \\
p = 1.00e-01 & PASA & 2.52e-11 & 1.31e-14 & 1.80e-16 & 7.7489646e-01& 1.47e+01 & \\
\hline
m = 500 & Algorithm 1 & 4.14e-11 & 3.52e-08 & 1.76e-11 & 1.8476054e+00 & 8.20e+00 & 85/68/23 \\
p = 3.00e-01 & PASA & 2.86e-17 & 1.34e-09 & 7.57e-14 & 1.8475701e+00& 4.45e+01 & \\
\hline
m = 1000 & Algorithm 1 & 8.55e-11 & 3.59e-07 & 2.34e-11 & 7.7506635e-01 & 1.74e+01 & 77/35/14 \\
p = 1.00e-01 & PASA & 1.43e-17 & 4.40e-15 & 4.25e-18 & 7.7506581e-01& 1.90e+02 & \\
\hline
m = 1000 & Algorithm 1 & 6.75e-13 & 7.31e-07 & 3.24e-10 & 1.8546431e+00 & 5.24e+01 & 185/260/38 \\
p = 3.00e-01 & PASA & 5.73e-18 & 4.51e-10 & 1.64e-14 & 1.8546203e+00& 5.43e+02 & \\
\hline
m = 1500 & Algorithm 1 & 5.61e-11 & 5.61e-07 & 1.75e-12 & 7.8224560e-01 & 7.33e+01 & 101/52/17 \\
p = 1.00e-01 & PASA & 9.21e-13 & 2.13e-13 & 1.09e-16 & 7.8224486e-01& 6.21e+02 & \\
\hline
m = 1500 & Algorithm 1 & 5.19e-11 & 7.69e-07 & 4.26e-12 & 1.8720793e+00 & 1.58e+02 & 237/348/67 \\
p = 3.00e-01 & PASA & 8.25e-12 & 5.04e-14 & 1.88e-17 & 1.8720697e+00& 2.21e+03 & \\
\hline
m = 2000 & Algorithm 1 & 9.01e-11 & 9.85e-07 & 7.59e-11 & 7.8735106e-01 & 1.39e+02 & 100/28/16 \\
p = 1.00e-01 & PASA & 3.06e-18 & 2.01e-15 & 4.96e-19 & 7.8734838e-01& 2.12e+03 & \\
\hline
m = 2000 & Algorithm 1 & 9.86e-11 & 8.33e-07 & 4.96e-12 & 1.8848184e+00 & 3.79e+02 & 161/209/58 \\
p = 3.00e-01 & PASA & - & - & - & -& - & \\
\hline
\end{longtable}
\end{tiny}
\end{center}
}

We compare our Algorithm 1 with \texttt{PASA} for solving the 
GW problem \eqref{GWDs} arising from graph matching problems. We did not include Gurobi because it cannot handle problems
with nonconvex objective functions. The results are presented in Table~\ref{TGWD}.
From the table, we can see that Algorithm 1 can solve all the problems to the required accuracy, while \texttt{PASA} cannot solve the last instance within 10 hours. Moreover, Algorithm 1 is  faster than \texttt{PASA} for larger problems with $m\geq 500$. {For this experiment, the projected gradient steps of Algorithm 1 is much greater than that of the previous two examples. The reason is that all the problems are degenerate and the number of nonzero entries is usually smaller than than the number of constraints.}
 Note that although we set the tolerance of \texttt{PASA} to be $10^{-6}$ using \texttt{pasadata.pasa.grad\_tol = 1e-6}, the KKT residuals of the outputs of \texttt{PASA} are usually close to machine precision. One reason is that this solver combines the projected gradient method and an active set method. Thus, once it identifies the correct active set of a certain stationary point, it can find that stationary point very accurately.


Apart from (GWD), our algorithm can also compute the following more complicated Gromov-Wasserstein berycenter (GWBC) problem:
\begin{equation}\label{GWBC}
\min\left\{ \sum_{i=1}^k \lambda_i {\rm GW}\(C,C_i,\mu,\nu_i\):\ C\in \R^{m\times m} \right\},
\end{equation}
where $\lambda_i\geq 0$ such that $\sum_{i=1}^k \lambda_i=1,$ $C_i\in \R^{m\times n_i}$,
$i\in[k]$ are given distance matrices,
 $\mu\in \R^m_+$, 
 $\nu_i\in \R^{n_i}_+$, $i\in[k]$, are given probability vectors.
  Problem (\ref{GWBC}) was proposed by Peyr{\'e} et al. in \cite{peyre2016gromov}, where it was applied in shape interpolation and quantum chemistry. In problem (\ref{GWDs}), the variable $C$ is the barycenter of the data matrices $C_1,\ldots,C_k.$ Notably, computing $ {\rm GW}\(C,C_i,\mu,\nu_i\)$ involves another optimization problem with variables $X_i\in \Pi\(\mu,\nu_i\)$ for $i\in[k]$. After combining the outer and inner problems, we get the following non-convex optimization problem over a convex polyhedron whose objective function is a cubic polynomial:
\begin{equation}\label{GWBCs}
\min\left\{ \sum_{i=1}^k \lambda_i\( \frac{\mu^\top (C\circ C)\mu+\nu_i^\top 
(C_i\circ C_i)\nu_i}{2}-\<X_i,CX_iC_i\> \) \,:\, 
 \begin{array}{l} 
 C\in \R^{m\times m}, \\
  X_i\in \Pi\(\mu,\nu_i\) \;\forall\; i\in[k] 
  \end{array}\right\}.
\end{equation}
Problem (\ref{GWDs}) is very challenging to solve and currently the state-of-the-art approach is
based on the block-coordinate minimization method, with its subproblems solved by a Sinkhorn-type method \cite{xu2019scalable}. Observe that when $X_i$'s are fixed, the variable $C$ has the closed form solution 
$C=\frac{1}{\mu\mu^\top} \circ \sum_{i=1}^k \lambda_iX_i C_i X_i^\top,$ where
$\frac{1}{\mu\mu^\top}$ denotes the elementwise reciprocal of $\mu\mu^\top$.
Thus we can substitute this formula into \eqref{GWBCs} to get  the following simplified problem
after some manipulations:

\begin{multline}\label{GWBCss}
\min\Bigg\{ \sum_{i=1}^k \frac{1}{2} \lambda_i\nu_i^\top (C_i\circ C_i)\nu_i
-\frac{1}{2} \(\frac{1}{\mu}\)^\top
\( \Big(\sum_{i=1}^k \lambda_iX_i C_i X_i^\top\Big) \circ \Big(\sum_{i=1}^k \lambda_iX_i C_i X_i^\top\Big) \)
\frac{1}{\mu}:\\ 
\mu \in \R^m_+,\;
X_i\in \Pi\(\mu,\nu_i\),\ i\in [k] \Bigg\}.
\end{multline}
In the above, $\frac{1}{\mu}$ denotes the elementwise reciprocal of $\mu$.
Problem (\ref{GWBCss}) is an optimization problem over a convex polyhedral set, so we can directly use our Riemannian based optimization method to find a stationary point of it. 
{Note that for \eqref{GWBCss}, the total dimension of the variables
is $m + m\sum_{i=1}^k n_i$, which can be quite large even just for $m=n_i=1000$ for all $i\in[k]$.}

One application of Gromov-Wasserstein barycenter is multi-graph matching problem, where we match a source graph to many target graphs simultaneously. This problem has been tested in \cite{xu2019scalable} using protein-protein interaction (PPI) network\footnote{The dataset is available on \href{https://www3.nd.edu/~cone/MAGNA++/}{https://www3.nd.edu/~cone/MAGNA++/}.} of yeast with its noisy version \cite{vijayan2015magna++}. In our experiment, we test the same problem on the same dataset. We match the original PPI network with its first $k$ ($k=2,3,4,5$) noisy versions with $5\%,10\%,15\%,20\%,25\%$ noisy edges. We set $\lambda_i=1/k$ and $\mu,v_i$ to be ${\bf 1}_n/n,$ where $n$ is the uniform size of all the networks. Because of the significant non-convexity of problem (\ref{GWBCss}), we find that the outputs of Algorithm 1 and \texttt{PASA} are sometimes quite different, making it difficult for us to compare their efficiency. Concerning this issue, we use the {\sc Matlab}
 solver based on block-coordinate minimization\footnote{Source codes from \href{https://github.com/gpeyre/2016-ICML-gromov-wasserstein}{https://github.com/gpeyre/2016-ICML-gromov-wasserstein}.} in \cite{peyre2016gromov} to warm-start Algorithm 1 and \texttt{PASA}. Because this solver doesn't have a name, we call it \texttt{BCM}. \texttt{BCM} solves an entropy-regularized version of (\ref{GWBC}), where we choose the regularization parameter to be $10^{-3}.$ We find that the parameter $10^{-3}$ works reasonably well and a smaller value will cause some numerical issue for \texttt{BCM}.  Note that \texttt{BCM} solves (\ref{GWBCs}) instead of (\ref{GWBCss}) and the multiplier is not returned. Therefore, we will only use the variable $X_i$'s from \texttt{BCM} and apply one step of projected gradient to recover the multiplier. We use the code \texttt{PPROJ} in \texttt{SuiteOPT} to compute the projection onto the polyhedral set. By doing this, we can show the output information of \texttt{BCM} to demonstrate that both Algorithm 1 and \texttt{PASA} further improve the solution of \texttt{BCM}.

{

\begin{center}
\begin{tiny}
\begin{longtable}{|c|c|cccccc|}
\caption{Comparison of Algorithm 1, \texttt{PASA}, and \texttt{BCM} for computing Gromov-Wasserstein barycenter in multi-graph matching problems. All the graphs have $1004$ vertices.}
\label{TGWBC}
\\
\hline
problem & algorithm & Rp & Rd & Rc& obj & time & iter/ldl/pg \\ \hline
\endhead
3 graphs& Algorithm 1 & 9.94e-11 & 4.69e-07 & 1.98e-09 & 1.8562315e-04& 1.73e+01 & 50/18/ 5 \\
& Sinkhorn & 1.20e-16 & 4.93e-05 & 7.59e-08 & 1.8725382e-04& 3.88e+01 & \\
& PASA & 2.29e-18 & 4.02e-10 & 1.05e-12 & 1.8562316e-04& 1.08e+02 & \\
\hline
4 graphs & Algorithm 1 & 1.11e-10 & 6.08e-07 & 1.46e-09 & 2.6363669e-04& 2.24e+01 & 60/24/ 6 \\
& Sinkhorn & 1.41e-16 & 4.01e-05 & 8.74e-08 & 2.6468280e-04& 1.13e+02 & \\
& PASA & 1.55e-18 & 1.48e-08 & 7.70e-15 & 2.6407071e-04& 1.07e+02 & \\
\hline
5 graphs & Algorithm 1 & 7.12e-11 & 2.15e-07 & 4.23e-10 & 3.4174061e-04& 2.70e+01 & 90/64/ 8 \\
& Sinkhorn & 2.86e-17 & 1.15e-04 & 5.64e-07 & 3.5201978e-04& 1.11e+02 & \\
& PASA & 2.53e-17 & 9.05e-07 & 5.15e-11 & 3.4231615e-04& 3.96e+02 & \\
\hline
6 graphs & Algorithm 1 & 1.33e-10 & 2.08e-07 & 5.28e-10 & 4.1902994e-04& 3.65e+01 & 93/75/ 8 \\
& Sinkhorn & 7.52e-18 & 9.79e-05 & 6.88e-07 & 4.2649783e-04& 1.27e+02 & \\
& PASA & 9.17e-18 & 4.85e-18 & 1.75e-19 & 4.1732142e-04& 2.46e+02 & \\
\hline
\end{longtable}
\end{tiny}
\end{center}
}

From Table~\ref{TGWBC}, we can see that both Algorithm 1 and \texttt{PASA} improve the outputs of \texttt{BCM}, which don't reach the accuracy of $10^{-6}.$ Algorithm 1 is significantly faster than \texttt{PASA} for all the examples. However, in terms of accuracy, neither Algorithm 1 nor \texttt{PASA} is consistently better than the other in that the function value of Algorithm 1 can be both slightly smaller or slightly larger than those of \texttt{PASA}. This is within our expectation because problem (\ref{GWBCss}) is {non-convex.} 


\section{Conclusion}\label{Sec-conc}
{In this work, we study the method of Hadamard parametrization for optimization problems 
over a convex polyhedron. In order to study the optimality conditions without the 
LICQ condition, we derive  explicit formulas of the tangent cone and second-order tangent set at a given point in the parametrized polyhedron, which is an algebraic variety. In order to extend Riemannian optimization method to solve the new problem on an algebraic variety, we propose a manifold decomposition method that decomposes the parametrized polyhedron into a finite disjoint union of Riemannian manifolds. We also design a globally convergent hybrid algorithm that adaptively chooses a submanifold and applies Riemannian gradient descent on it. {Our hybrid algorithm uses a projected gradient method to jump from one submanifold to another when the Riemannian gradient descent method stalls or when a nearly singular iterate is encountered.} Numerical experiments are conducted to demonstrate the efficacy
of our algorithm as compared to other state-of-the-art algorithms.
Our algorithm is a first-order method that only uses the function value and gradient information of the objective function. In the future, we will study how to incorporate second-order information in our algorithmic framework. Also, currently, we have not {established} any convergence rate for our algorithm due to the non-smooth nature of the algebraic variety. It would be interesting to study the convergence rate of optimization algorithms on an algebraic variety. 
}

\appendix

%
%
%
%
%

\section{Some useful results}

\begin{lemma}\label{LDL}
Suppose $A\in \R^{n\times n}$ is a symmetric positive semidefinite matrix, then it has an LDL decomposition $A=L\DD(d)L^\top$ such that $L$ is a lower triangular matrix with $\dd(L)=e,$ and $d\in \R^n_+.$ Moreover, the vector $d$ is unique.
\end{lemma}

\begin{proof}
We only prove the uniqueness of $d$ because the existence of LDL decomposition of $A$ comes from Theorem 2.2.6 in \cite{so2007semidefinite}. Suppose $A$ has two LDL decompositions, i.e.,$A=L_1\DD(d_1)L_1^\top=L_2\DD(d_2)L_2^\top.$ Because $L_1,L_2$ have diagonal entries being $1$'s, they are invertible. We have that 
\begin{equation}\label{linalg1}
L_2^{-1}L_1\DD(d_1)=\DD(d_2)L_2^\top L_1^{-\top},
\end{equation}
 where $L_2^{-1}L_1$ is a lower triangular matrix such that $\dd\(L_2^{-1}L_1\)=e$ and $L_2^\top L_1^{-\top}$ is an upper triangular matrix such that $\dd( L_2^\top L_1^{-\top} )=e.$
  From (\ref{linalg1}), we know that $d_1=\dd(L_2^{-1}L_1\DD(d_1))=\dd(\DD(d_2)L_2^\top L_1^{-\top})=d_2.$
\end{proof}

\begin{lemma}\label{sigmaregular}
For any $\sigma>0$ and $y\in \PA$ with regular representation $(M,N,r)$ and $x=y\circ y,$ the set $\M_\sigma(y)$ defined in Definition~\ref{defi-sing} is closed.
\end{lemma}
\begin{proof}
Consider any $\{y_k\}_{k\in \mathbb{N}^+}\in \M_\sigma(y)$ such that $\lim_{k\rightarrow \infty} y_k=y^*.$ Let $x_k=y_k\circ y_k$ for any $k\in \mathbb{N}^+$ and $x^*=y^*\circ y^*.$ Let $L_k\DD(d_k)L_k^\top$ be the LDL decomposition of $A\DD(x_k)A^\top$ with $d_k \geq 0.$ Without loss of generality, assume that for any $j\in [m]\setminus \supp(d_k)$ and $i>j,$ $L_k(i,j)=0.$ This is because the entries of $d_k$ in $[m]\setminus\supp(d_k)$ are zero's, so the corresponding columns of $L_k$ doesn't influence the matrix multiplication. With this assumption, we have that 
\begin{eqnarray}
&& \<I,A\DD(x_k)A^\top\>=\< I,L_k\DD(d_k)L_k^\top \>
\nonumber \\
&\geq& \sigma \(\<I,L_kL_k^\top\>-(m-r)\)=\sigma\(\|L_k\|^2-(m-r)\).
\label{Sep_24_lem_1}
\end{eqnarray}
 Because $x_k\rightarrow x^*,$ $\{ \<I,A\DD(x_k)A^\top\> \}_{k\in \mathbb{N}^+}$ is bounded. From (\ref{Sep_24_lem_1}), we know that $\{L_k\}_{k\in \mathbb{N}^+}$ is also bounded. Moreover, because $\dd\(L_k\)=e,$ we have that $\< I,L_k\DD(d_k)L_k^\top \>\geq e^\top d_k.$ Thus, $\{d_k\}_{k\in \mathbb{N}^+}$ is also bounded. Hence there exists a subsequence $\{i_k\}_{k\in \mathbb{N}^+}$ such that $d_{i_k}\rightarrow d^*$, $L_{i_k}\rightarrow L^*$ and we have that $A\DD(x^*)A^\top=L^*\DD(d^*)L^{*\top}.$ 
 
It is easy to see that $\supp(d^*)=r$ and the smallest nonzero entries of $d^*$ is bounded below by $\sigma.$ Also, $\dd(L^*)=e.$ Moreover,
$N \dd(L^*\DD(d^*)L^{*\top})=\lim_{k\rightarrow\infty}N L_{i_k}\DD(d_{i_k})L_{i_k}^{\top}=0.$ Therefore, we have that $y^*\in \M_\sigma(y).$
\end{proof}

\begin{lemma}\label{lemconnec}
Consider $A\in \R^{m\times n}.$ Define the bipartite graph $H:=\([m]\times [n],\E\)$ whose biadjacency matrix corresponds to the underlying sparsity pattern of $A,$ that is, $(i,j)\in \E$ if and only if $A_{ij}\neq 0.$ Consider the linear operator $\Xi:\R^m\times \R^n\rightarrow \R^{m\times n}$ such that $\Xi(\lambda,\mu)=\DD(\lambda)A+A\DD(\mu).$ If $H$ is connected, then $\Ker\left[ \Xi \right]={\rm span}\( ({\bf 1}_m,-{\bf 1}_n) \).$
\end{lemma}

\begin{proof}
Because $H$ is connected, it has no isolated vertices, which means that $A$ has no zero rows or columns. Because $\DD({\bf 1}_m)A+A\DD(-{\bf 1}_n)={\bf 0}_{m\times n},$ we have that $({\bf 1}_m,-{\bf 1}_n)\in \Ker\left[ \Xi \right].$ 
Now assume on the contrary that $\Ker\left[ \Xi \right]\neq {\rm span}\( ({\bf 1}_m,-{\bf 1}_n) \)$. Then there exists $\(\lambda,\mu\)\in \R^m\times \R^n$ such that $\DD(\lambda)A+A\DD(\mu)=0$ and $\(\lambda,\mu\)\notin {\rm span}\( ({\bf 1}_m,-{\bf 1}_n) \).$ Without loss of generality, assume that $\(\lambda,\mu\)$ contain some zero entries; if not,  
we can consider a suitable linear combination of $\(\lambda,\mu\)$ and $({\bf 1}_m,-{\bf 1}_n) $ to cancel some entries. 
Let $\I_1=\supp(\lambda),$ $\I_2=[m]\setminus \I_1$ and $\J_1\in \supp(\mu),$ $\J_2=[n]\setminus \J_1.$ Because $\Xi\( \lambda,\mu \)={\bf 0}_{m\times n},$ we have that

\begin{equation}\label{23Nov2_1}
\begin{bmatrix}
\DD(\lambda_{\I_1})A_{\I_1,\J_1}& \DD(\lambda_{\I_1})A_{\I_1,\J_2}\\
{\bf 0}_{|\I_2|\times |\J_1|}& {\bf 0}_{|\I_2|\times |\J_2|}
\end{bmatrix}
+
\begin{bmatrix}
A_{\I_1,\J_1}\DD(\mu_{\J_1})& {\bf 0}_{|\I_1|\times |\J_2|}\\
A_{\I_2,\J_1}\DD(\mu_{\J_1})&{\bf 0}_{|\I_2|\times |\J_2|}
\end{bmatrix}=0.
\end{equation}
We will next proceed to show that  $\I_1,\I_2,\J_1,\J_2$ are all nonempty.
Because $\( \lambda,\mu \)$ contains zero entries, without loss of generality, we may assume that $\I_2\neq \emptyset.$ Now, if $\J_2=\emptyset,$ then $\J_1=[n].$ From (\ref{23Nov2_1}), we get $A_{\I_2,:}={\bf 0}_{|\I_2|\times n},$ which contradicts the fact that $A$ has no zero rows. Therefore, $\J_2\neq\emptyset.$ 
Moreover, 
because $(\lambda,\mu)$ is nonzero, without loss of generality, we may assume that $\I_1\neq \emptyset.$ If $\J_1=\emptyset,$ then $\J_2=[n].$ From (\ref{23Nov2_1}), we get $A_{\I_1,:}={\bf 0}_{|\I_1|\times n},$ which again 
contradicts  the fact that $A$ has no zero rows. Thus, $\J_1\neq \emptyset.$ Therefore we have proved that $\I_1,\I_2,\J_1,\J_2$ are all nonempty. 

From (\ref{23Nov2_1}), we can see that $A_{\I_1,\J_2}={\bf 0}_{|\I_1|\times |\J_2|}$ and $A_{\I_2,\J_1}={\bf 0}_{|\I_2|\times |\J_1|}.$ This implies that $H$ is not connected, and we get a contradiction.
\end{proof}

\section{Proofs of some results}

\subsection{Proof of Proposition~\ref{Prop-tcone}}

\begin{proof}
Since $\T_{\PA}^i(y)\subset \T_{\PA}(y),$ we only have to prove that $\T_{\PA}(y)\subset \widehat{\T}(y)\subset \T_{\PA}^i(y).$

\noindent{\bf Step 1.} $\T_{\PA}(y)\subset \widehat{\T}(y).$ For any $h\in \T_{\PA}(y),$ from (\ref{CTcone}), we have that there exists $\{t_k\}_{\mathbb{N}^+}\subset \R_{+}$ such that $t_k\downarrow 0$ and $\dist(y+t_kh,\PA)=o(t_k).$ Consider $\{w_k\}_{\mathbb{N}^+}\subset \R^n$ such that $\norm{w_k}=o(t_k)$ and $y+t_kh+w_k\in \PA.$ Then we have that 
\begin{eqnarray}
& MA\( (y+t_kh+w_k)\circ (y+t_kh+w_k)  \)-Mb=0 &
\label{Oct_25_1}
\\[3pt]
& NA\( (y+t_kh+w_k)\circ (y+t_kh+w_k)  \)=0. &
\label{Oct_25_2}
\end{eqnarray}
Expanding (\ref{Oct_25_1}), we get $2t_kMA\( y\circ h  \)+o(t_k)=0$, which implies that $MA\( y\circ h  \)=0.$ Expanding (\ref{Oct_25_2}) and using $NA\Diag{y}=0$, we get $t_k^2NA\( h\circ h \)+o(t_k^2)=0,$ which implies that $NA\( h\circ h \)=0.$ Thus, $h\in \widehat{\T}(y)$ and so $\T_{\PA}(y)\subset \widehat{\T}(y).$

\noindent{\bf Step 2.} $\widehat{\T}(y)\subset \T^i_{\PA}(y).$ For any $h\in \widehat{\T}(y)$, since $w\rightarrow MA\( y\circ w \)$ is surjective, we can choose $w\in \R^n$ such that $MA\(h\circ h+y\circ w\)=0.$ Consider the mapping $\gamma:\R \rightarrow \R^n$ such that 
\begin{equation}\label{Oct_25_5}
\gamma(t):=y\circ y+2ty\circ h+t^2 h\circ h+t^2y\circ w.
\end{equation}
From $MA\(h\circ h+y\circ w\)=0$ and $MA\(y\circ h\)=0,$ we have that $MA\gamma(t)=Mb.$ From $NA\(h\circ h\)=0$ and $NA\Diag{y}=0,$ we have that $NA\gamma(t)=0=Nb.$ Thus, $[M;N]A\gamma(t)=[M;N]b$ and so $A\gamma(t)=b.$ From (\ref{Oct_25_5}), there exists $\delta>0$ such that for any $t\in [-\delta,\delta],$ $\gamma(t)\in \R^n_+.$ This implies that for any $t\in [-\delta,\delta],$ $\sqrt{\gamma(t)}\in \PA$, where the square-root denotes coordinate-wise operation. Simple Taylor expansion shows that $\sqrt{\gamma(t)}=y+th+o(t).$ Therefore, $\dist(y+th,\PA)\leq \dist(y+th,\sqrt{\gamma(t)}) =o(t)$ and $h\in \T_{\PA}^i(y).$ This means that $\widehat{\T}(y)\subset \T^i_{\PA}(y).$
\end{proof}

\subsection{Proof of Proposition~\ref{Prop-2tset}}

\begin{proof}
Since $\T^{i,2}_{\PA}(y,h)\subset \T^2_{\PA}(y,h),$ we only have to prove that $\T^2_{\PA}(y,h)\subset \widehat{\T^2}(y,h)\subset \T^{i,2}_{\PA}(y,h).$

\noindent{\bf Step 1.} $\T^2_{\PA}(y,h)\subset \widehat{\T^2}(y,h).$
For any $w\in \T^2_{\PA}(y,h),$ from (\ref{O2tset}), we have that there exists $\{t_k\}_{\mathbb{N}^+}\subset \R_+$ such that $t_k\downarrow 0$ and $\dist\(y+t_kh+t_k^2w/2,\PA\)=o(t_k^2).$ Consider $z_k\in \R^n$ such that $\norm{z_k}=o(t_k^2)$ and $y+t_kh+t_k^2w/2+z_k\in \PA.$ Then we have that
\begin{eqnarray}
&& MA\( \(y+t_kh+t_k^2w/2+z_k\)\circ \(y+t_kh+t_k^2w/2+z_k\)  \)-Mb=0,
\label{Oct_25_6}
\\[3pt]
&& VNA\( \(y+t_kh+t_k^2w/2+z_k\)\circ \(y+t_kh+t_k^2w/2+z_k\)  \)=0,
\label{Oct_25_7}
\\[3pt]
&& WNA\( \(y+t_kh+t_k^2w/2+z_k\)\circ \(y+t_kh+t_k^2w/2+z_k\)  \)=0.
\label{Oct_25_8}
\end{eqnarray}
Expanding (\ref{Oct_25_6}) and using $MA(y\circ h)=0,$ we get $t_k^2MA\( h\circ h+y\circ w \)+o(t_k^2)=0$, which implies that $MA\( h\circ h+y\circ w \)=0.$ Expanding (\ref{Oct_25_7}) and using $NA\Diag{y}=0,$ $NA(h\circ h)=0,$ we get $t_k^3 VNA \( h\circ w \)+o(t_k^3)=0,$ which implies that $VNA \( h\circ w \)=0.$ Expanding (\ref{Oct_25_8}) and using $NA\Diag{y}=0,$ $WNA\Diag{h}=0$ (coming from $VW^\top=0$), we get $t_k^4W\(w\circ w\)/4+o(t_k^4)=0,$ which implies that $WNA\(w\circ w\)=0.$ Therefore, we have that $w\in \widehat{\T^2}(y,h)$ and so $\T^2_{\PA}(y,h)\subset \widehat{\T^2}(y,h)$

\noindent{\bf Step 2.} $\widehat{\T^2}(y,h) \subset \T^{i,2}_{\PA}(y,h).$
Consider $w\in \widehat{\T^2}(y,h).$ Because $MA\Diag{y}$ has full row rank, $VNA\Diag{y}=0$ and $VNA\Diag{h}$ has full row rank, we have that $[MA\Diag{y};VNA\Diag{h}]$ has full row rank. Therefore, we may choose $z\in \R^n$ such that 
\begin{equation}\label{Oct_27_1}
MA\( h\circ w+2y\circ z \)=0,\ VNA\(w\circ w/4+2h\circ z\)=0.
\end{equation}
Because $MA\Diag{y}$ has full row rank, there exists $q\in \R^n$ such that
\begin{equation}\label{Oct_27_2}
MA\( w\circ w/4+2h\circ z+2y\circ q \)=0.
\end{equation}
With above chosen vectors $z$ and $q,$ define the mapping $\gamma:\R\rightarrow \R^n$ such that
\begin{equation}
\gamma(t):=y\circ y+2ty\circ h+t^2\(h\circ h+y\circ w\)+t^3\(h\circ w+2y\circ z\)+t^4\(w\circ w/4+2h\circ z+2y\circ q\),\notag
\end{equation}
which comes from the fourth-order Taylor expansion of the following function of 
$t$:
$$\(y+th+{t^2w}/{2}+t^3z+t^4q\)\circ \(y+th+{t^2w}/{2}+t^3z+t^4q\).$$
From $MA\( y\circ h \)=0,$ $MA\( h\circ h+y\circ w \)=0,$ (\ref{Oct_27_1}) and (\ref{Oct_27_2}), we have that $MA\gamma(t)=Mb.$
From $NA\Diag{y}=0,$ $NA\(h\circ h\)=0,$ $VNA\( h\circ w \)=0$ and (\ref{Oct_27_1}), we have that $VNA\gamma(t)=0.$
 From $WNA\Diag{y}=0,$ $WNA\dd(h)=0$ and $WNA(w\circ w)=0,$ we have that $ WNA\gamma(t)=0.$ Combining $MA\gamma(t)=Mb,$ $VNA\gamma(t)=0$ and $WNA\gamma(t)=0,$ we have $[M;VN;WN]A\gamma(t)=[M;VN;WN]b.$ By using the fact that $[M;N]$ and $[V;W]$ are full rank, we have that $[M;VN;WN]$ is full rank and 
 so $A\gamma(t)=b.$ After some algebraic manipulations, we have that
 \begin{equation}\label{2023_Sep18_1}
 \gamma(t)=\(y+th+{t^2}w/2\)\circ \(y+th+{t^2}w/2\)+2t^3 y\circ z+2t^4 h\circ z+2t^4 y\circ q.
 \end{equation}
From (\ref{2023_Sep18_1}), there exists $\delta>0$ such that for any $t\in [-\delta,\delta],$ $\gamma(t)\in \R^n_+.$ This implies that for any $t\in [-\delta,\delta],$ $\sqrt{\gamma(t)}\in \PA.$ Simple Taylor expansion shows that
  \begin{equation}
 \sqrt{\gamma(t)}=\(y+th+{t^2}w/2\)+O(t^3).
 \end{equation}
 Therefore, $\dist(y+th+t^2w/2,\PA) \leq \dist(y+th+t^2w/2,\sqrt{\gamma(t)}) = o(t^2)$. Hence
 $w\in \T^{i,2}_{\PA}(y,h)$ and so $\widehat{\T^2}(y,h) \subset \T^{i,2}_{\PA}(y,h).$
\end{proof}

\subsection{Proof of Proposition~\ref{biparprop}}

\begin{proof}
{\bf Sufficiency} Assume on the contrary that $\BG(Y)$ is not connected. Because $Y$ has no zero columns and rows, there exists decompositions $\I_1\sqcup \I_2=\J_1\sqcup \J_2=[n]$ such that $\I_1,\J_1,\I_2,\J_2\neq \emptyset $ and $\BG(Y)$ has no edges in $\I_1\times \J_2$ or $\I_2\times \J_1.$ This implies that $Y_{\I_1,\J_2}={\bf 0}_{|\I_1|\times |\J_2|}$ and $Y_{\I_2,\J_1}={\bf 0}_{|\I_2|\times |\J_1|}.$ Let $X=Y\circ Y,$ from the definition of $\B_n,$ we have that $|\I_1|=|\I_1|$ and $|\I_2|=|\J_2|.$ Moreover, $X_{\I_1,\J_1}\in \B_{|\I_1|},$ $X_{\I_2,\J_2}\in \B_{|\I_2|},$ which implies that $Y_{\I_1,\J_1}\in \BH_{|\I_1|},$ $Y_{\I_2,\J_2}\in \BH_{|\I_2|}.$ The Jacobian of the constraints in $\BH_n$ at $Y$ is the linear operator 
\begin{equation}\label{Jaco}
{\bf J}_Y: H\rightarrow \left[ 2\dd(YH^\top);2\dd\(Y^\top H\)_{1:n-1} \right],
\end{equation}
whose adjoint mapping is as follows:
\begin{equation}\label{adJaco}
{\bf J}_Y^*: [\lambda;\mu]\rightarrow  2\DD(\lambda)Y+2Y\DD\([\mu;0]\).
\end{equation}
Without loss of generality, assume that $n\in \J_2.$ Now, we choose $\lambda\in \R^n$ and $\mu\in \R^{n-1}$ such that $\lambda_{\I_1}={\bf 1}_{|\I_1|},$ $\lambda_{\I_2}={\bf 0}_{|\I_2|},$ $\mu_{\J_1}=-{\bf 1}_{|\J_1|}$ and $\mu_{\J_2\setminus \{n\}}={\bf 0}_{|\J_2|-1}.$ Substituting $[\lambda;\mu]$ into (\ref{adJaco}) and using the fact that $Y_{\I_1,\J_2}={\bf 0}_{|\I_1|\times |\J_2|}$ and $Y_{\I_2,\J_1}={\bf 0}_{|\I_2|\times |\J_1|}.$ We have that ${\bf J}_Y^*\( [\lambda;\mu] \)={\bf 0}_{n\times n}.$ Because $[\lambda;\mu]\neq 0,$ we know that ${\bf J}_Y^*$ is not injective and so $Y$ is singular. Thus we get a contradiction.\\

\noindent{\bf Necessity} Assume on the contrary that $Y$ is singular. Then there exists $(\lambda,\mu)\in \R^n\times \R^{n-1}$ such that $[\lambda;\mu]\neq {\bf 0}_{2n-1}$ and ${\bf J}_Y^*([\lambda;\mu])=0.$ Let $\I_1:=\supp(\lambda),$ $\I_2:=[n]\setminus \I$ and $\J_1:=\supp(\mu),$ $\J_2:=[n-1]\setminus \J_1.$ From (\ref{adJaco}), we have the following equation:
\begin{equation}
\begin{bmatrix}
\DD(\lambda_{\I_1})Y_{\I_1,\J_1}&\DD(\lambda_{\I_1})Y_{\I_1,\J_2}&\DD(\lambda_{\I_1}) Y_{\I_1,n}\\
{\bf 0}_{|\I_2|\times |\J_1|} & {\bf 0}_{|\I_2|\times |\J_2|} & {\bf 0}_{|\I_2|}
\end{bmatrix}
+
\begin{bmatrix}
Y_{\I_1,\J_1}\DD(\mu_{\J_1})& {\bf 0}_{|\I_1|\times |\J_2|}& {\bf 0}_{|\I_1|}\\
Y_{\I_2,\J_1}\DD(\mu_{\J_1})& {\bf 0}_{|\I_2|\times |\J_2|}& {\bf 0}_{|\I_2|}
\end{bmatrix}=0,\notag
\end{equation}
from which we can easily see that 
\begin{equation}\label{Jacodis}
[Y_{\I_1,\J_2},Y_{\I_1,n}]={\bf 0}_{|\I_1|\times \(|\J_2|+1\)},\ Y_{\I_2,\J_1}={\bf 0}_{|\I_2|\times |\J_1|}. 
\end{equation}
Now we consider two cases.

\noindent{\bf Case 1.} $\J_1=\emptyset.$
 We have that $\J_2=[n-1].$ From (\ref{Jacodis}), we have that $Y_{\I_1,:}={\bf 0}_{|\I_1|\times n}.$ If $\I_1\neq \emptyset,$ then we have a contradiction because $Y$ doesn't have zero rows. Thus, we have that $\I_1=\emptyset.$ This implies that $[\lambda;\mu]={\bf 0}_{2n-1},$ which is a contradiction.
 
 \noindent{\bf Case 2.} $\J_1\neq \emptyset.$
 If $\I_1=\emptyset,$ then $\I_2=[n].$ From (\ref{Jacodis}), we have that $Y_{:,\J_1}={\bf 0}_{n\times |\J_1|},$ which contradicts the fact that $Y$ doesn't have zero columns. If $\I_2=\emptyset,$ then $\I_1=[n].$ From (\ref{Jacodis}), we have that $Y_{:,n}={\bf 0}_n,$ which also contradicts to that $Y$ doesn't have zero columns. Therefore, $\I_1,\I_2,\J_1,\J_2\cup \{ n\}$ are all nonempty. From (\ref{Jacodis}), $\BG(Y)$ is disconnected, and we get a contradiction. 
\end{proof}

\subsection{Proof of Proposition~\ref{manidecB}}

\begin{proof}
Without loss of generality assume that $n\in \J_k.$ Since $X:=Y\circ Y\in \B_n,$ we have that for any $i\in [k],$ $X_{\I_i,\J_i}\in \B_{|\I_i|}$ and thus $|\I_i|=|\J_i|.$ 

\noindent{\bf Case 1.} $Y$ is smooth. 

According to Proposition~\ref{biparprop}, $\BG(Y)$ is connected. This means that $k=1.$ In this case, $Z\in \M(Y)$ is equivalent to that $Z$ is smooth. This is further equivalent to that $\BG(Z)$ is connected, i.e., it has the same connected components as $\BG(Y).$ 

\noindent{\bf Case 2.}  $Y$ is singular. 

According to Proposition~\ref{biparprop}, $k\geq 2.$ Let $(\lambda,\mu)\in \R^{n}\times \R^{n-1}.$ From (\ref{adJaco}) and sparsity pattern of $\BG(Y)$, we know that ${\bf J}_Y^*([\lambda;\mu])=0$ is equivalent to that the following conditions hold:
\begin{itemize}
\item $\forall i\in [k-1],\ \DD(\lambda_{\I_i})Y_{\I_i,\J_i}+Y_{\I_i,\J_i}\DD(\mu_{\J_i})
={\bf 0}_{|\I_i|\times |\J_i|}$,
\item $ \DD(\lambda_{\I_k})Y_{\I_k,\J_k}+Y_{\I_k,\J_k}\DD([\mu_{\J_k\setminus\{n\}};0])=0$.
\end{itemize}
Because for any $i\in [k],$ $\BG(Y_{\I_i,\J_i})$ is connected, according to Lemma~\ref{lemconnec}, the above conditions are further equivalent to that

\begin{itemize}
\item $\forall i\in [k-1],\ \(\lambda_{\I_i},\mu_{\J_i}\)\in {\rm span}\( \({\bf 1}_{|\I_i|},-{\bf 1}_{|\I_i|}\) \)$,
\item $\lambda_{\I_k}={\bf 0}_{|\I_k|},$ $\mu_{\J_k\setminus \{n\}}={\bf 0}_{|\I_k|-1}$.
\end{itemize}
Note that the vectors $(\lambda,\mu)\in \R^{n}\times \R^{n-1}$ satisfying the above conditions is a linear subspace $\mathcal{LS}(Y)$ in $\R^{n\times n}$. From Definition~\ref{Regrep}, we can see that $(M,N,r)$ is a regular representation of $Y$ if and only if: (1) the row vectors of $N$ is a basis of $\mathcal{LS}(Y)$; (2) the row vectors of $M$ is a basis of the orthogonal complement of $\mathcal{LS}(Y)$ in $\R^{n\times n}$; (3) $r=2n-1-{\rm dim}\( \mathcal{LS}(Y) \)=2n-k.$ Note that all these properties are determined by 
the linear space $\mathcal{LS}(Y),$ which is determined by the connected components of $\BH(Y).$ Therefore, we have that $Z\in \M(Y)$ if and only if it has the same connected components as $Y.$
\end{proof}

\subsection{Proof of Proposition~\ref{gdprop}}

\begin{proof}
From the continuous differentiability of $\phi$ and smoothness of $\M(y),$ we have that $\g f (y)$ is also continuous on $\M(y).$ Moreover, there exists $\beta_{y,1},\ c_{y,1}>0$ such that for any $z\in \M(y)\cap B_{\beta_{y,1}}(y)$ and any $t\in [0,c_{y,1}]$, there following three inequalities hold:
\begin{equation}\label{Sep_24_0}
\int_0^1 \|\gamma'(s)\|{\rm d}s\leq 2\| \Re\(z,-t \g f(z)\)-z \|\leq 4t\|\g f(y)\|,
\end{equation}
where $\gamma:[0,1]\rightarrow \M(y)$ is the geodesic from $z$ to $\Re\(z,-t \g f(z)\)$;
\begin{equation}\label{Sep_24_0.1}
\sup_{s\in [0,1]}\{\| \g f(\gamma(s))-\g f(y) \|\}\leq (1-\delta)\| \g f(y)\|/8;
\end{equation}
\begin{eqnarray}
&& \< \g f(y),\Re\(z,-t \g f(z)\)-z +t \g f(y)\>
\nonumber \\
&=&\< \g f(y),\Re\(z,-t \g f(z)\)-z +t \g f(z)\>+\< \g f(y),t\g f(y)-t\g f(z)\>
\nonumber \\
&\leq & t(1-\delta)\| \g f(y)\|^2/2.
\label{Sep_24_0.2}
\end{eqnarray}
Note that we have used the fact that $\|\g f(y)\|\neq 0$ to ensure the above three inequalities. Therefore, we have that
\begin{eqnarray}
&& \hspace{-7mm} f\( \Re\(z,-t \g f(z)\) \)-f(z)=\int_0^1 \< \g f(\gamma(s)),\gamma'(s) \>{\rm d}s
\nonumber \\
&=& \int_0^1\< \g f(\gamma(s))-\g f(y),\gamma'(s) \>{\rm d}s+\int_0^1\< \g f(y),\gamma'(s) \>{\rm d}s
\nonumber \\
&\leq& \int_0^1 \|\g f(\gamma(s))-\g f(y)\|\|\gamma'(s)\|{\rm d}s
+\< \g f(y),\Re\(z,-t \g f(z)\)-z \>
\nonumber \\
&\leq & 
t (1-\delta)\| \g f(y) \|^2/2+t(1-\delta)\| \g f(y)\|^2/2-t \| \g f(y) \|^2
\nonumber \\
&=& -t\delta \| \g f(y)\|^2,
\label{Sep_24_1}
\end{eqnarray}
where the second inequality comes from the inequalities (\ref{Sep_24_0}), (\ref{Sep_24_0.1}) and (\ref{Sep_24_0.2}).
\end{proof}

\subsection{Proof of Proposition~\ref{PGprop}}

\begin{proof}
From the continuity of $\nabla \phi(x)$ and the projection mapping $\P_{\C_{A,b}}(\cdot),$ there exists $\beta_{y,2}$, $c_{y,2}>0$ such that for any $u\in \C_{A,b}\cap B_{\beta_{y,2}}(x)$ and $t\in (0,c_{y,2}],$ the following inequality holds:
\begin{equation}\label{Sep_24_2}
\sup_{s\in [0,1]}\(\| \nabla \phi(\gamma(s))-\nabla \phi(u) \|\)\leq (1-\delta)\| g_t(u) \|/t,
\end{equation}
where $\gamma(s)= (1-s)u+s\P_{\C_{A,b}}\( u-t \nabla \phi(u) \).$ Note that we have used $\|g_t(u)\|\neq 0$ to get (\ref{Sep_24_2}). Then, we have that
\begin{eqnarray*}
&& \hspace{-7mm}
\phi\( \P_{\C_{A,b}}\( u-t \nabla \phi(u) \) \)-\phi(u)=\phi\(u+g_t(u)\)-\phi(u)=\int_0^1  \< \nabla \phi(\gamma(s)),\gamma'(s) \>{\rm d}s
\nonumber \\
&=& \int_0^1  \< \nabla \phi(\gamma(s))-\nabla\phi(u),\gamma'(s) \>{\rm d}s+\int_0^1  \< \nabla\phi(u),\gamma'(s) \>{\rm d}s
\\
&\leq& \int_0^1  \|\nabla \phi(\gamma(s))-\nabla\phi(u)\|\|\gamma'(s) \|{\rm d}s
+\< \nabla\phi(u),g_t(u) \>
\nonumber \\
& \leq & \sup_{s\in [0,1]}\(\| \nabla \phi(\gamma(s))-\nabla \phi(u) \|\)\| g_t(u) \|-\| g_t(u) \|^2/t
\nonumber \\
&\leq& -\delta\| g_t(u) \|^2/t=-t\cdot \delta\| h_t(u) \|^2,
\end{eqnarray*}
where the second inequality comes from Proposition~\ref{PGdefi} and the third   from (\ref{Sep_24_2}).
\end{proof}

\bibliographystyle{abbrv}
\bibliography{LCP}

\end{document}